\pdfoutput=1 
\DeclareSymbolFont{AMSb}{U}{msb}{m}{n}
\documentclass[11pt,noamsfonts,a4paper]{amsart}
\usepackage[left=1in, right=1in, top=1in, bottom=1in]{geometry}
\usepackage{mathtools}
\usepackage{braket}
\usepackage[charter]{mathdesign}
\usepackage[tracking]{microtype}
\usepackage{tabularx}
\usepackage{etoolbox}
\usepackage{graphicx}
\usepackage{stmaryrd} 

\usepackage{amsmath,amsthm}
\newtheoremstyle{pineapple}%
  {1em}{1em}%
  {\itshape}{}%
  {\bfseries}{. ---}
  {0.5em}{}

\newtheoremstyle{durian}%
  {1em}{1em}%
  {}{}%
  {\bfseries}{. ---}
  {0.5em}{}

\makeatletter
\def\swappedhead#1#2#3{%
  \thmnumber{\@upn{\the\thm@headfont#2\@ifnotempty{#1}{.~}}}%
  \thmname{#1}%
  \thmnote{ {\the\thm@notefont(#3)}}}
\makeatother

\makeatletter
\def\@sect#1#2#3#4#5#6[#7]#8{%
  \edef\@toclevel{\ifnum#2=\@m 0\else\number#2\fi}%
  \ifnum #2>\c@secnumdepth \let\@secnumber\@empty
  \else \@xp\let\@xp\@secnumber\csname the#1\endcsname\fi
  \@tempskipa #5\relax
  \ifnum #2>\c@secnumdepth
    \let\@svsec\@empty
  \else
    \refstepcounter{#1}%
    \edef\@secnumpunct{%
      \ifdim\@tempskipa>\z@ 
        \@ifnotempty{#8}{.~}%
      \else
        \@ifempty{#8}{.}{.~}%
      \fi
    }%
    \@ifempty{#8}{%
      \ifnum #2=\tw@ \def\@secnumfont{\bfseries}\fi}{}%
    \protected@edef\@svsec{%
      \ifnum#2<\@m
        \@ifundefined{#1name}{}{%
          \ignorespaces\csname #1name\endcsname\space
        }%
      \fi
      \@seccntformat{#1}%
    }%
  \fi
  \ifdim \@tempskipa>\z@ 
    \begingroup #6\relax
    \@hangfrom{\hskip #3\relax\@svsec}{\interlinepenalty\@M #8\par}%
    \endgroup
    \ifnum#2>\@m \else \@tocwrite{#1}{#8}\fi
  \else
  \def\@svsechd{#6\hskip #3\@svsec
    \@ifnotempty{#8}{\ignorespaces#8\unskip
       \@addpunct.}%
    \ifnum#2>\@m \else \@tocwrite{#1}{#8}\fi
  }%
  \fi
  \global\@nobreaktrue
  \@xsect{#5}}
\makeatother

\makeatletter
\def\@seccntformat#1{%
  \protect\textup{\protect\@secnumfont
    \ifnum\pdfstrcmp{subsection}{#1}=0 \bfseries\fi
    \csname the#1\endcsname
    \protect\@secnumpunct
  }%
}
\makeatother

\theoremstyle{pineapple}
\newtheorem{IntroTheorem}{Theorem}

\swapnumbers
\newtheorem{Theorem}[subsection]{Theorem}
\newtheorem{Lemma}[subsection]{Lemma}
\newtheorem{Proposition}[subsection]{Proposition}
\newtheorem{Corollary}[subsection]{Corollary}

\theoremstyle{durian}

\usepackage[dvipsnames]{xcolor}
\usepackage{hyperref} 
\hypersetup{
  colorlinks = true,
  linkcolor = Fuchsia,
  urlcolor = Fuchsia,
  citecolor = ForestGreen,
  linkbordercolor = {white}}
\usepackage{tikz-cd}
\usetikzlibrary{arrows}

\tikzset{
  symbol/.style={
    draw=none,
    every to/.append style={
      edge node={node [sloped, allow upside down, auto=false]{$#1$}}}
  }
}

\usepackage{enumitem}
\setlist[1]{labelindent=\parindent}
\setlist[1]{labelsep=0.5em}
\setlist[enumerate,1]{label={\upshape (\roman*)}, ref={\upshape (\roman*)}}

\makeatletter
\newcommand{\leqnomode}{\tagsleft@true\let\veqno\@@leqno}
\newcommand{\reqnomode}{\tagsleft@false\let\veqno\@@eqno}
\makeatother

\usepackage[utf8]{inputenc}

\tikzset{>={Straight Barb[length=2pt,width=4pt]}, commutative diagrams/arrow style=tikz}

\makeatletter
\let\c@equation\c@subsection

\makeatother

\DeclareMathOperator{\CH}{CH}
\DeclareMathOperator{\gr}{gr}

\DeclareMathOperator{\Fr}{Fr}
\DeclareMathOperator{\id}{id}
\DeclareMathOperator{\Spec}{Spec}

\DeclareMathOperator{\image}{im}
\DeclareMathOperator{\HomSheaf}{\mathcal{H}\!\mathit{om}}

\newcommand{\genstirlingI}[3]{%
  \genfrac{[}{]}{0pt}{#1}{#2}{#3}%
}

\newcommand{\stirlingI}[2]{\genstirlingI{}{#1}{#2}}

\makeatletter
\newcommand*{\coloneqq}{\mathrel{\rlap{%
           \raisebox{0.3ex}{$\m@th\cdot$}}%
           \raisebox{-0.3ex}{$\m@th\cdot$}}%
           =}
\newcommand{\eqqcolon}{=%
           \mathrel{\rlap{%
           \raisebox{0.3ex}{$\m@th\cdot$}}%
           \raisebox{-0.3ex}{$\m@th\cdot$}}}
\makeatother

\newcommand{\parref}[1]{{\bf\ref{#1}}}

\DeclareMathOperator{\Gram}{Gram}
\DeclareMathOperator{\rank}{rank}

\DeclareMathOperator{\Hom}{Hom}

\DeclareMathOperator{\codim}{codim}
\DeclareMathOperator{\pr}{pr}

\DeclareMathOperator{\Sing}{Sing}
\DeclareMathOperator{\Fitt}{Fitt}

\DeclareMathOperator{\Gr}{\mathbf{Gr}}

\newcommand{\kk}{\mathbf{k}}
\newcommand{\FF}{\mathbf{F}}
\newcommand{\GG}{\mathbf{G}}
\newcommand{\sO}{\mathcal{O}}
\newcommand{\qbics}{q\operatorname{\bf\!-bics}}
\newcommand{\hrefSP}[1]{\href{https://stacks.math.columbia.edu/tag/#1}{#1}}
\newcommand{\citeSP}[1]{\cite[\hrefSP{#1}]{stacks-project}}
\newcommand{\citeForms}[1]{\cite[\href{https://arxiv.org/pdf/2301.09929.pdf\#subsection.#1}{\textbf{#1}}]{qbic-forms}}
\newcommand{\citeThesis}[1]{\cite[\href{https://arxiv.org/pdf/2205.05273.pdf\#subsection.#1}{\textbf{#1}}]{thesis}}

\makeatletter
\newcommand{\smallbullet}{} 
\DeclareRobustCommand\smallbullet{%
  \mathord{\mathpalette\smallbullet@{0.75}}%
}
\newcommand{\smallbullet@}[2]{%
  \vcenter{\hbox{\scalebox{#2}{$\m@th#1\bullet$}}}%
}
\makeatother

\newcommand{\subsectiondash}[1]{\subsection{#1}\textbf{---}\;}
\newcommand{\PP}{\mathbf{P}}

\title{\(q\)-bic hypersurfaces and their Fano schemes}
\author{Raymond Cheng}
\address{Institute of Algebraic Geometry \\
  Leibniz University Hannover \\
  Welfengarten 1 \\
  30167 Hannover \\
  Germany
}
\email{cheng@math.uni-hannover.de}

\begin{document}
\begin{abstract}
A \emph{\(q\)-bic hypersurface} is a hypersurface in projective space of degree
\(q+1\), where \(q\) is a power of the positive ground field characteristic,
whose equation consists of monomials which are products of a \(q\)-power and a
linear power; the Fermat hypersurface is an example. I identify \(q\)-bics as
moduli spaces of isotropic vectors for an intrinsically defined bilinear form
and use this to study their Fano schemes of linear spaces. Amongst other
things, I prove that the scheme of \(m\)-planes in a smooth
\((2m+1)\)-dimensional \(q\)-bic hypersurface is an \((m+1)\)-dimensional
smooth projective variety of general type which admits a purely inseparable
covering by a complete intersection; I compute its Betti numbers by relating it
to Deligne--Lusztig varieties for the finite unitary group; and I prove that
its Albanese variety is purely inseparably isogenous via an Abel--Jacobi map to
a intermediate Jacobian of the hypersurface. The case \(m = 1\) may be viewed
as an analogue of results of Clemens and Griffiths regarding cubic threefolds.
\end{abstract}
\maketitle
\setcounter{tocdepth}{1}

\thispagestyle{empty}
\section*{Introduction}
Explicitly, a \emph{\(q\)-bic hypersurface} is any hypersurface in
projective space of degree \(q+1\), where \(q\) is a power of the
characteristic \(p > 0\) of the ground field \(\kk\), defined by an equation of
the special form:
\[
X \coloneqq
\Set{(x_0: \cdots : x_n) \in \PP^n | \sum\nolimits_{i,j = 0}^n a_{ij} x_i^q x_j = 0}.
\]
Such hypersurfaces, the best known being the Fermat hypersurface of degree
\(q+1\), have been considered time and time again for their extraordinary
properties and interconnections: they are supersingular \cite{Tate:Conjecture,
SK:Fermat}, unirational though not necessarily separably so \cite{Shioda:Fermat,
Shioda:Unirationality, Conduche, Shen:Fermat, ratconn}, have curious
differential-geometric properties \cite{Wallace:Duality, Hefez:Thesis,
Beauville:Moduli, Noma, KP:Gauss}, and are otherwise extremal
in many respects, such as in relation to \(F\)-singularity theory
\cite{KKPSSW:F-Pure}. These hypersurfaces arise in the study of
Deligne--Lusztig theory \cite{Lusztig:Green, DL, Li:DL}, unitary
Shimura varieties \cite{Vollaard, LZ:Kudla, LTXZZ}, and finite Hermitian
geometries \cite{BC:Hermitian, Segre:Hermitian, Hirschfeld:Geometries}. My goal
here is to develop a new perspective from which to understand these
hypersurfaces, bringing new methods to bear and new analogies to make sense of
their idiosyncrasies.

To explain, begin with an old observation: the \(q\)-bic hypersurface \(X\) is
defined by a bilinear form. More precisely, view \(\PP^n\) as the space of
lines in a vector space \(V\), and let \(e_0,\ldots,e_n\) be the basis dual to
the chosen coordinates. The equation for \(X\) is intrinsically
encoded by the biadditive pairing \(\beta \colon V \times V \to \kk\)
determined by \(\beta(e_i,e_j) = a_{ij}\), \(\kk\)-linearity in the second
variable, and \(q\)-power \(\kk\)-linearity in the first. Geometric
properties of \(X\) are reflected in algebraic properties of \(\beta\): for
instance, \(X\) is smooth over \(\kk\) if and only if \(\beta\) is nonsingular,
equivalent to invertibility of the matrix \((a_{ij})_{i,j=0}^n\).

The basic point of this article is to identify \(X\) as the moduli space of
isotropic lines for \(\beta\). This brings moduli- and deformation-theoretic
methods to bear, and begins an analogy with quadrics. This perspective
immediately highlights the schemes parameterizing linear subvarieties of \(X\):
they become moduli spaces of isotropic subspaces for \(\beta\) and are thus
akin to orthogonal Grassmannian. Their basic geometry is as follows:

\begin{IntroTheorem}\label{theorem-fano-schemes}
For each \(0 \leq r < \frac{n}{2}\), the Fano scheme \(\FF\) of \(r\)-planes
in a \(q\)-bic hypersurface \(X \subset \PP^n\) is
\begin{enumerate}
\item nonempty;
\item of dimension at least \((r+1)(n-2r-1)\);
\item connected when \(n \geq 2r+2\); and
\item smooth of dimension \((r+1)(n-2r-1)\) precisely at points corresponding
to \(r\)-planes disjoint from the singular locus of \(X\).
\end{enumerate}
In particular, if \(X\) is smooth, then \(\FF\) is smooth of dimension
\((r+1)(n-2r-1)\) and is irreducible whenever \(\FF\) is positive-dimensional.
Furthermore, in this case, \(\FF\)
\begin{enumerate}
\setcounter{enumi}{4}
\item is stratified by generalized Deligne--Lusztig varieties of type \({}^2
\mathrm{A}_{n+1}\); and
\item
has canonical bundle
\(\omega_\FF \cong \sO_\FF\big((r+1)(q+1) - (n+1)\big)\)
where \(\sO_\FF(1)\) is its Pl\"ucker polarization.
\end{enumerate}
\end{IntroTheorem}

This is the cumulation of \parref{basics-fano-equations},
\parref{hypersurfaces-nonempty-fano},
\parref{differential-generic-smoothness}, \parref{differential-singular-locus},
\parref{hypersurfaces-fano-connected}, and \parref{dl-fano-schemes-are-dl}.
The basic geometric properties are established by moduli-theoretic means;
notably, connectedness is proven by studying the relative Fano schemes over the
parameter space of all \(q\)-bic hypersurfaces. The tangent sheaf of \(\FF\)
can be identified: see \parref{differential-identify}. Perhaps the most
remarkable fact is that the dimension of \(\FF\) is independent of \(q\),
equivalently, of the degree of \(X\): see the comments following
\parref{basics-fano-equations} for more.

I believe that the Fano schemes of smooth \(q\)-bic hypersurfaces are an
interesting collection of smooth projective varieties, with rich and
fascinating geometry, worthy of further study; the remaining results, which
focus on the Fano schemes of maximal-dimensional planes, hopefully give some
substance to this conviction. I mention in passing two other particularly
interesting cases which I hope will be taken up in future work: first, the
schemes of lines for their analogy with cubics: and, second, the scheme of
\(r\)-planes in a smooth \(q\)-bic hypersurface of dimension \(n - 1 =
(q+1)(r+1)-2\). In the latter case, \(\FF\) is a smooth projective variety of
dimension \((r+1)^2(q-1)\) with trivial canonical bundle, and is, furthermore,
simply connected as soon as \(q > 2\): see
\parref{hypersurfaces-simply-connected}.

Turning now to the remaining results, consider first a smooth \(q\)-bic
hypersurface \(X\) of even dimension \(2m\). Its Fano scheme \(\FF\) of
\(m\)-planes is a finite set of reduced points, the number of which being
geometrically determined in \parref{hermitian-maximal-count} as
\[
\#\FF = \prod\nolimits_{i = 0}^m (q^{2i+1} + 1).
\]
Taking \(m = 1\), this means that a smooth \(q\)-bic surface contains exactly
\((q+1)(q^3+1)\) lines; specializing further to the case \(q = 2\) recovers the
\(3 \times 9 = 27\) lines in a smooth cubic surface.

When \(X\) is of odd dimension \(2m+1\), the geometry of \(\FF\) is more
complicated, and its basic properties are summarized in the following
statement. Below, the \'etale Betti numbers of \(\FF\) are expressed in terms
of Gaussian binomial coefficients with parameter \(\bar{q} \coloneqq -q\): see
\parref{hermitian-planes-count} for the notation and comments on the choice of
parameter.

\begin{IntroTheorem}\label{theorem-lagrangian}
The Fano scheme \(\FF\) of \(m\)-planes in a smooth \(q\)-bic hypersurface
\(X \subset \PP^{2m+2}\) is an \((m+1)\)-dimensional, irreducible, smooth,
projective variety of general type. There exists a dominant purely inseparable
rational map \(Z \dashrightarrow \FF\) of degree \(q^{m(m+1)}\) from a complete
intersection \(Z \subset \PP^{2m+2}\) geometrically isomorphic to
\[
Z \cong
\Set{(x_0:x_1: \cdots : x_{2m+2}) \in \PP^{2m+2} |
x_0^{q^{2i+1} + 1} + x_1^{q^{2i+1} + 1} + \cdots + x_{2m+2}^{q^{2i+1} + 1} = 0\;\;
\text{for}\; 0 \leq i \leq m}.
\]
For each \(0 \leq k \leq 2m+2\), the \(k\)-th \'etale Betti number of \(\FF\)
is given by
\[
b_k(\mathbf{F}) =
\bar{q}^{\binom{2m+3-k}{2}} \stirlingI{2m+2}{k}_{\bar{q}} +
\sum_{i = 0}^{m - \lceil  k/2\rceil} (1-\bar{q})^{i+1}
\bar{q}^{\binom{2m+1-2i-k}{2}}
[2i+1]_{\bar{q}}!!
\stirlingI{2m+3}{2i+2}_{\bar{q}}
\stirlingI{2m-2i}{k}_{\bar{q}}.
\]
\end{IntroTheorem}

The latter two statements here are a combination of \parref{cyclic-resolve},
\parref{cyclic-maximal}, and \parref{dl-betti}. The Betti numbers are computed
via Deligne--Lusztig theory from \cite{DL}, using results of Lusztig from
\cite{Lusztig:Frobenius}, and as set up in \parref{dl-maximal-isotropics}. The
covering is related to a unitary analogue of the wonderful compactification of
Drinfeld's upper half-space, see \parref{cyclic-resolve}, and arises from
dynamics of, essentially, the Frobenius endomorphism \(\phi\) of \(X\): as
explained in \parref{hermitian-filtration}, the complete intersection \(Z\)
parameterizes the isotropic lines \(L \subset V\) such that the cyclic subspace
spanned by \(L, \phi(L), \ldots, \phi^r(L)\) remains isotropic for \(\beta\);
and that \(Z\) covers \(\FF\) means that the general \(r\)-plane in \(X\) is
cyclically generated by \(\phi\).

For the final result, notice that Poincar\'e duality for \(\FF\) yields a
series of interesting identities amongst Gaussian numbers and, in particular,
gives a simple formula for the first Betti number of \(\FF\):
\[
b_1(\FF) = b_{2m+1}(\FF) = \bar{q} [2m+2]_{\bar{q}}.
\]
This number coincides with the middle Betti number \(b_{2m+1}(X)\) of the
hypersurface, see \parref{dl-hypersurface-cohomology}, leading to a remarkable
geometric consequence: the Albanese variety \(\mathbf{Alb}_\FF\) of \(\FF\) is
essentially isomorphic to the intermediate Jacobian \(\mathbf{Ab}_X^{m+1}\) of
\(X\). Here, the intermediate Jacobian of \(X\) is taken to mean the algebraic
representative, in the sense of Samuel and Murre from \cite{Samuel,
Murre:Jacobian}, of the group of algebraically trivial cycles of codimension
\(m+1\) in \(X\): see \parref{cgaj-algrep}. The main statement is
as follows:

\begin{IntroTheorem}\label{theorem-cgaj}
Over an algebraically closed field \(\kk\), the Fano incidence correspondence
\(\FF \leftarrow \mathbf{L} \rightarrow X\) induces a purely inseparable
isogeny
\[
\mathbf{L}_* \colon
\mathbf{Alb}_\FF \to
\mathbf{Ab}_X^{m+1}
\]
between supersingular abelian varieties of dimension
\(\frac{1}{2}\bar{q}[2m+2]_{\bar{q}}\).
\end{IntroTheorem}

This is the essential statement contained in
\parref{cgaj-intermediate-jacobian-morphisms} and \parref{cgaj-result}. Its
proof is based on studying special subvarieties in \(\FF\), using the inductive
structure provided by the moduli-theoretic view of \(X\). Besides suggesting
curious phenomena regarding Chow motives in positive characteristic, this
result provides the first explicit series of nontrivial algebraic
representatives in the literature.

Taken all together, the case \(m = 1\) of \(q\)-bic threefolds is particularly
striking: A smooth \(q\)-bic threefold \(X\) has a smooth surface \(S\) of
lines, and the Albanese variety of the surface \(S\) is essentially isomorphic
to the intermediate Jacobian of the hypersurface \(X\) via the Abel--Jacobi
map. Phrased in this way, \(q\)-bic threefolds are analogous to complex cubic
threefolds when compared with the results of Clemens and Griffiths in
\cite{CG}. In fact, when \(q = 2\) and \(X\) is a cubic threefold in
characteristic \(2\), this result plays a crucial role in establishing
irrationality of \(X\) in a situation where all previous methods do not apply.
The analogies run deeper, as can already be seen upon comparing the
geometry in the latter half of \S\parref{section-cgaj} with that in
\cite[Chapter 5]{Huybrechts:Cubics}; the comments following
\parref{cgaj-result} indicate a bit more, including analogies relating to Prym
varieties. The companion paper \cite{qbic-threefolds} studies the geometry of
\(S\) in more detail.

I would like to end this Introduction with three comments on the name
\emph{\(q\)-bic}: First, that these hypersurfaces deserve a uniform name
for varying \(q\) recognizes that they ought to be studied together, to be
understood as a single family, being characterized by the intrinsic structure
given by the form \(\beta\). Second, being moduli
spaces of isotropic vectors, these hypersurfaces behave in many ways like
\emph{q}uadrics. Third, other aspects of their geometry, for instance
in regards to the geometry of lines, evoke that of other low-degree
hypersurfaces, notably of cubic hypersurfaces. I hope the results presented
here and in the future will convince the reader that the naming is worthwhile,
and that the name, apt.

\medskip
\noindent\textbf{Relations with other works. --- }
As previously indicated, \(q\)-bic hypersurfaces have been studied by many
mathematicians in many different contexts: see the comments of
\cite[pp.7--11]{thesis} for a brief survey. Particularly relevant include:
Shimada's work in \cite{Shimada:Lattices}, which is perhaps the first to
systematically study the geometry of linear spaces in the hypersurface via
the associated bilinear form \(\beta\); the works \cite{Tate:Conjecture,
Lusztig:Green, Vollaard, LTXZZ, Li:DL}, in which schemes related to the Fano
schemes considered here arise in relation to Deligne--Lusztig theory; the
work \cite{KKPSSW:F-Pure}, which distinguishes \(q\)-bic hypersurfaces---called
\emph{extremal hypersurfaces} there---amongst all degree \(q+1\) hypersurfaces
via \(F\)-singularity theory; and \cite{BPRS:Lines}, which studies the
geometry of lines on smooth \(q\)-bic surfaces.

\medskip
\noindent\textbf{Outline. --- }
The basic formalism and technique used to study \(q\)-bic hypersurfaces in this
paper is set up in \S\parref{section-basics}. Smoothness and connectedness
properties of the Fano schemes are studied in
\S\parref{section-smoothness-and-connectedness}. Properties of the tautological
incidence correspondences amongst the Fano schemes are studied in
\S\parref{section-correspondences}. Hermitian structures, related to an intrinsic
\(\FF_{q^2}\)-structure on a smooth \(q\)-bic hypersurface, are studied in
\S\parref{section-hermitian}. This gives rise to the connection with
Deligne--Lusztig theory in \S\parref{section-dl}. Finally, special subvarieties in
the Fano variety \(\FF\) of maximal-dimensional linear subvarieties in a smooth
\(q\)-bic hypersurface \(X\) of odd dimension are studied in
\S\parref{section-cgaj}, leading to the relationship between the Albanese of
\(\FF\) and the intermediate Jacobian of \(X\).

\medskip
\noindent\textbf{Acknowledgements. --- }
This paper is based and greatly expands on Chapter 2 and parts of Chapter 4
of my thesis \cite{thesis}. Many thanks to Aise Johan de Jong for conversations
and continued interest regarding this work over the years. Thanks also to Jason
Starr, Chao Li, and Orsola Tommasi for various comments and questions that have
helped shape this work. Special thanks goes to Fumiaki Suzuki for pointing
out \cite{Fakhruddin:Supersingular}, providing existence of the intermediate
Jacobian in this case. During the initial stages of this work, I was partially
supported by an NSERC Postgraduate Scholarship.

\section{Setup}\label{section-basics}
As indicated in the Introduction, \(q\)-bic hypersurfaces are construed in this
paper as moduli spaces for a certain bilinear form. This Section develops this
perspective in an invariant way via the theory of \(q\)-bic forms, as developed
in \cite{qbic-forms}: their essential definitions and properties are recalled
in \parref{basics-qbic-forms} and a few further properties that will be used in
this paper are discussed in
\parref{hermitian-structures}--\parref{hermitian-minimal}. Then \(q\)-bic
hypersurfaces are defined in \parref{basics-definition} invariantly in terms
of \(q\)-bic forms. Their Fano schemes are introduced in
\parref{basics-fano-schemes}, and some of their basic properties are discussed
in the remainder of the Section.

Throughout this article, \(p\) is a prime, \(q \coloneqq p^\nu\) is a
positive integer power, and \(\kk\) is a field containing the finite field
\(\FF_{q^2}\). Given a finite-dimensional vector space \(V\) over \(\kk\),
write \(\PP V\) for the projective space of lines in \(V\). A \emph{plane}
refers to a linear subvariety of projective space.

\subsectiondash{\texorpdfstring{\(q\)}{q}-bic forms}\label{basics-qbic-forms}
To set notation, let \(R\) be an \(\FF_{q^2}\)-algebra and \(M\) a finite
projective \(R\)-module. Write \(M^\vee\) for its \(R\)-linear dual. Let
\(\Fr \colon R \to R\) be the \(q\)-power Frobenius homomorphism of \(R\) and
for integers \(i \geq 0\), define the \emph{\(i\)-th Frobenius twists} of \(M\)
by \(M^{[i]} \coloneqq R \otimes_{\Fr^i,R} M\). There is a canonical
\(q^i\)-linear map \((-)^{[i]} \colon M \to M^{[i]}\) given by
\(m \mapsto m^{[i]} \coloneqq 1 \otimes m\).

A \emph{\(q\)-bic form} on \(M\) over \(R\) is an \(R\)-linear map
\(\beta \colon M^{[1]} \otimes_R M \to R\). The form \(\beta\) induces two
adjoint maps, abusively denoted by \(\beta \colon M \to M^{[1],\vee}\) and
\(\beta^\vee \colon M^{[1]} \to M^\vee\). The form is called
\emph{nondegenerate} if its adjoint maps are injective and \emph{nonsingular}
if they are isomorphisms. An element \(m \in M\) is called \emph{isotropic}
for \(\beta\) if \(\beta(m^{[1]},m) = 0\). A subset \(N \subseteq M\)
is called \emph{isotropic} if every element in \(N\) is so.

Given submodules \(N_1 \subseteq M\) and \(N_2 \subseteq M^{[1]}\), write
\[
N_1^\perp \coloneqq \ker(\beta^\vee \colon M^{[1]} \to M^\vee \twoheadrightarrow N_1^\vee)
\quad\text{and}\quad
N_2^\perp \coloneqq \ker(\beta \colon M \to M^{[1],\vee} \twoheadrightarrow N_2^\vee)
\]
for their orthogonals with respect to \(\beta\). The orthogonals of \(M^{[1]}\)
and \(M\) are called the \emph{kernels} of \(\beta\); when the image of
\(\beta \colon M \to M^{[1],\vee}\) is a local direct summand, the kernels
fit into a canonical exact sequence of finite projective modules
\[
0 \to
M^{[1],\perp} \to
M \xrightarrow{\beta}
M^{[1],\vee} \to
M^{\perp,\vee} \to
0.
\]
The \emph{radical} of \(\beta\) is the largest submodule of \(M^{[1],\perp}\)
whose Frobenius twist lies in \(M^\perp\); equivalently,
\[
\operatorname{rad}(\beta) \coloneqq
\Set{m \in M | \beta(n^{[1]},m) = \beta(m^{[1]},n) = 0\;
\text{for all}\; n \in M}.
\]
For each \(i \geq 1\), twisting by \(\Fr^i\) yields a \(q\)-bic form
\(\beta^{[i]} \colon M^{[i+1]} \otimes_R M^{[i]} \to R\), characterized by
\[
\beta^{[i]}(m^{[i]}, n^{[i]}) = \beta(m,n)^{q^i}
\quad\text{for all}\;
m \in M^{[1]}
\;\text{and}\;
n \in M.
\]

\subsectiondash{Canonical endomorphisms}\label{hermitian-structures}
From now on, consider \(R = \kk\) a field and \(M = V\) a finite-dimensional
\(\kk\)-vector space. A nonsingular \(q\)-bic form \(\beta\) on \(V\) induces
two canonical endomorphisms: On the one hand, the adjoint map \(\beta \colon V
\to V^{[1],\vee}\) yields an isogeny of the \(\kk\)-algebraic group
\(\mathbf{GL}_V\) of linear automorphisms of \(V\), given by
\[
F \colon \mathbf{GL}_V \to \mathbf{GL}_V
\qquad
g \mapsto \beta^{-1} \circ g^{[1],\vee,-1} \circ \beta.
\]
Its fixed subgroup scheme \(\mathbf{GL}_V^F = \mathrm{U}(V,\beta)\) is called
the \emph{unitary group} of \((V,\beta)\), see \citeForms{5.6}. On the other
hand, the Frobenius-twist of \(\beta^\vee \colon V^{[1]} \to V^\vee\) together
with the inverse of \(\beta\) yields an isomorphism of \(\kk\)-vector spaces
\[
\sigma_\beta \coloneqq
\beta^{-1} \circ \beta^{[1],\vee}
\colon V^{[2]} \to V.
\]
This provides a descent datum for \(V\) to \(\mathbf{F}_{q^2}\)
and the associated absolute Frobenius morphism for this structure is given by
the canonical \(q^2\)-linear bijection
\[
\phi \coloneqq
\sigma_\beta \circ (-)^{[2]} \colon
V \to V^{[2]} \to V.
\]
The two endomorphisms are related: the square of \(F\) is the Frobenius
morphism of \(\mathbf{GL}_V\) induced by the \(\mathbf{F}_{q^2}\)-rational
structure \(\sigma_\beta\). In short, \(F^2 = \phi\).

Recall from \citeForms{2.1} that \(v \in V\) is called \emph{Hermitian} if it
satisfies \(\beta(u^{[1]}, v) = \beta(v^{[1]}, u)^q\) for all \(u \in V\).
The endomorphism \(\phi \colon V \to V\) generalizes this to a certain
symmetry property for \(\beta\):

\begin{Lemma}\label{hermitian-phi-properties}
\(\beta(w, \phi(v)) = \beta^{[1]}(v^{[2]},w)\)
for every \(v \in V\) and \(w \in V^{[1]}\), whence
\begin{enumerate}
\item\label{hermitian-phi-properties.fixed}
\(v\) is Hermitian if and only if \(v\) is fixed by \(\phi\);
\item\label{hermitian-phi-properties.double}
\(\beta(\phi(v_1)^{[1]}, \phi(v_2)) = \beta(v_1^{[1]},v_2)^{q^2}\) for
every \(v_1,v_2 \in V\); and
\item\label{hermitian-phi-properties.isotropic}
\(v\) is isotropic if and only if \(\phi(v)\) is isotropic.
\end{enumerate}
\end{Lemma}

\begin{proof}
It follows from the definition of \(\phi\) that
\[
\beta(w,\phi(v))
= w^\vee \circ \beta \circ (\beta^{-1} \circ \beta^{[1],\vee} \circ v^{[2]})
= w^\vee \circ \beta^{[1],\vee} \circ v^{[2]}
= \beta^{[1]}(v^{[2]},w).
\]
The remaining properties now follow from this identity.
\end{proof}

\subsection{Hermitian subspaces}\label{hermitian-subspaces}
Subspaces \(U\) of \(V\) fixed by \(\phi\) are called \emph{Hermitian
subspaces}. They are characterized by the following equivalent conditions:
\begin{enumerate}
\item\label{hermitian-subspaces.fixed}
\(\phi(U) = U\);
\item\label{hermitian-subspaces.U}
\(U\) is spanned by Hermitian vectors upon base change to the separable
closure of \(\kk\); and
\item\label{hermitian-subspaces.perps}
\(U^{[1],\perp,[1]} = U^\perp\) as subspaces of \(V^{[1]}\).
\end{enumerate}

\begin{proof}[Proof of equivalences]
For \ref{hermitian-subspaces.fixed} \(\Rightarrow\)
\ref{hermitian-subspaces.U}, assume already that \(\kk\) is separably
closed. Then \(\phi \colon U \to U\) is a bijective \(q^2\)-linear map, so
\(U\) has a basis consisting of fixed vectors, see
\cite[Expos\'e XXII, 1.1]{SGAVII}. This is a basis of Hermitian vectors by
\parref{hermitian-phi-properties}\ref{hermitian-phi-properties.fixed}.

For \ref{hermitian-subspaces.U} \(\Rightarrow\) \ref{hermitian-subspaces.perps},
it suffices to consider the case when \(U = \langle u \rangle\) is spanned
by a single Hermitian vector, wherein definitions of orthogonals, as in
\citeForms{1.7}, and of Hermitian vectors imply that
\[
U^{[1],\perp,[1]} =
\set{v \in V^{[1]} | \beta^{[1]}(u^{[2]},v) = 0} =
\set{v \in V^{[1]} | \beta(v,u) = 0}
= U^\perp.
\]

For \ref{hermitian-subspaces.perps} \(\Rightarrow\) \ref{hermitian-subspaces.fixed},
note that the equality of the orthogonals is equivalent to
\[
\image(\beta \colon U \subset V \to  V^{[1],\vee}) =
\image(\beta^{[1],\vee} \colon U^{[2]} \subset V^{[2]} \to V^{[1],\vee}).
\]
Thus \(\beta(v,U) = \beta^{[1]}(U^{[2]},v)\) for every
\(v \in V^{[1]}\) and the result follows upon comparing with
\parref{hermitian-phi-properties}.
\end{proof}

In particular, \parref{hermitian-subspaces}\ref{hermitian-subspaces.fixed}
identifies canonical Hermitian subspaces associated with any given subspace:

\begin{Corollary}\label{hermitian-minimal}
For any subspace \(U \subseteq V\),
\begin{align*}
\pushQED{\qed}
\bigcap\nolimits_{i \geq 0} \phi^i(U) & = \text{maximal Hermitian subspace contained in \(U\)},\;\text{and} \\
\sum\nolimits_{i \geq 0} \phi^i(U)  & = \text{minimal Hermitian subspace containing \(U\)}.
\qedhere
\popQED
\end{align*}
\end{Corollary}

\subsectiondash{\texorpdfstring{\(q\)}{q}-bic hypersurfaces}\label{basics-definition}
The theory of \(q\)-bic forms will now be used to define the geometric objects
of interest. Namely, given a nonzero \(q\)-bic form \((V,\beta)\) over \(\kk\),
the associated \emph{\(q\)-bic hypersurface} is the subscheme of \(\PP V\)
parameterizing isotropic lines in \(V\) for \(\beta\); in other words, this is
\[
X \coloneqq X_\beta \coloneqq \Set{[v] \in \PP V | \beta(v^{[1]},v) = 0}.
\]
Comparing with the moduli description of projective space, as described in
\citeSP{01NE}, shows that the \(q\)-bic hypersurface \(X\) represents the
functor
\(\mathrm{Sch}_{\mathbf{k}}^{\mathrm{opp}} \to \mathrm{Set}\) given by
\[
S \mapsto \Set{\mathcal{V}' \subset V_S \coloneqq V \otimes_\kk \sO_S \;
\text{a subbundle of rank \(1\) isotropic for \(\beta\)}}.
\]

\subsectiondash{}\label{basics-equations}
The \(q\)-bic form \(\beta\) induces a specific equation for \(X\): Write
\(\mathrm{eu} \colon \sO_{\PP V}(-1) \to V_{\PP V}\) for the tautological
line subbundle on \(\PP V\). The \emph{\(q\)-bic polynomial} associated with
\(\beta\) is the section
\[
f_\beta \coloneqq \beta(\mathrm{eu}^{[1]},\mathrm{eu}) \colon
\sO_{\PP V}(-q-1) \to
V_{\PP V}^{[1]} \otimes_{\sO_{\PP V}} V_{\PP V} \xrightarrow{\beta} \sO_{\PP V}.
\]
Contemplating the moduli description of \(\PP V\) shows that
\(X_\beta = \mathrm{V}(f_\beta)\). To recover the explicit description from
the Introduction, choose a basis \(V = \langle e_0,\ldots,e_n\rangle\) and let
\(\mathbf{x}^\vee \coloneqq (x_0:\cdots:x_n)\) be the corresponding
coordinates on \(\PP V = \PP^n\). For each \(0 \leq i, j \leq n\),
let \(a_{ij} \coloneqq \beta(e_i^{[1]},e_j)\) be the \((i,j)\)-entry of the
associated Gram matrix, as in
\cite[\href{https://arxiv.org/pdf/2301.09929.pdf\#subsection.1.2}{\textbf{1.2}}]{qbic-forms}.
Then the \(q\)-bic polynomial \(f_\beta\) is the homogeneous polynomial
\[
f_\beta(x_0,\ldots,x_n)
= \mathbf{x}^{\vee,[1]} \cdot \Gram(\beta;e_0,\ldots,e_n) \cdot \mathbf{x}
= \sum\nolimits_{i,j = 0}^n a_{ij} x_i^q x_j.
\]

\subsectiondash{Classification}\label{basics-classification}
Isomorphic \(q\)-bic forms yield projectively equivalent \(q\)-bic
hypersurfaces. Over an algebraically closed field, \cite[Theorem A]{qbic-forms}
shows that there are finitely many isomorphism classes of \(q\)-bic forms on a
given vector space in that, given a \(q\)-bic form \(\beta\) on \(V\), there
exists a basis \(V = \langle e_0,\ldots,e_n \rangle\) and nonnegative integers
\(a, b_m \in \mathbf{Z}_{\geq 0}\) such that
\[
\Gram(\beta;e_0,\ldots,e_n) =
\mathbf{1}^{\oplus a} \oplus
\Big(\bigoplus\nolimits_{m \geq 1} \mathbf{N}_m^{\oplus b_m}\Big)
\]
where \(\mathbf{1}\) is the \(1\)-by-\(1\) matrix with unique entry \(1\),
\(\mathbf{N}_m\) is the \(m\)-by-\(m\) Jordan block with \(0\)'s on the
diagonal, and \(\oplus\) denotes block diagonal sums of matrices. The tuple
\((a; b_m)_{m \geq 1}\) is the fundamental invariant of \(\beta\), hence
of \(X\), and is called its \emph{type}; the sum \(\sum_{m \geq 1} b_m\)
is its \emph{corank}. Note that the \(\mathbf{N}_1^{\oplus b_1}\)
piece is a \(b_1\)-by-\(b_1\) zero matrix and its underlying subspace is the
radical of \(\beta\). From this, it is clear that \(X\) is a cone if and only
if \(b_1 \neq 0\) and that its vertex is given by the radical of \(\beta\); see
also \citeThesis{2.4} for an invariant treatment.

A linear section of a \(q\)-bic hypersurface is another \(q\)-bic hypersurface:

\begin{Lemma}\label{basics-hyperplane-section}
Let \(X\) be the \(q\)-bic hypersurface associated with a \(q\)-bic form
\((V,\beta)\) and let \(U \subseteq V\) be any linear subspace. Then
\(X \cap \PP U\) is the \(q\)-bic hypersurface associated with the \(q\)-bic
form
\[
\beta_U \colon
U^{[1]} \otimes_\kk U \subset
V^{[1]} \otimes_\kk V \xrightarrow{\beta} \kk,
\]
which is possibly zero, obtained by restricting \(\beta\) to \(U\).
\qed
\end{Lemma}

\subsectiondash{Fano schemes}\label{basics-fano-schemes}
For each \(0 \leq r \leq n\), let \(\GG \coloneqq \GG(r+1,V)\) denote the
Grassmannian parameterizing \((r+1)\)-dimensional subspaces of \(V\), and let
\[
\FF \coloneqq
\FF_r(X) \coloneqq
\Set{[U] \in \GG | \PP U \subseteq X}
\]
denote the Fano scheme parameterizing \(r\)-planes of \(\PP V\) contained in
\(X\); see, for example, \cite{AK:Fano} or
\cite[Section V.4]{Kollar:RationalCurves} for generalities. Comparing
\parref{basics-definition} and \parref{basics-hyperplane-section} shows that
a \(q\)-bic form \((V,\beta)\) defining \(X\) endows \(\FF\) with
the alternative moduli interpretation as the moduli of \((r+1)\)-dimensional
isotropic subspaces for \(\beta\); in other words, \(\FF\) represents
the functor \(\mathrm{Sch}^{\mathrm{opp}}_\kk \to \mathrm{Set}\) given by
\[
S \mapsto
\Set{\mathcal{V}' \subset V_S \;\text{a subbundle of rank \(r+1\) isotropic for \(\beta\)}}.
\]
In particular, this implies that the equations of \(\FF\) in
\(\GG\) may be given in terms of the
tautological subbundle \(\mathcal{S} \subseteq V_{\GG}\) of
rank \(r+1\) as follows:

\begin{Lemma}\label{basics-fano-equations}
Let \(X\) be the \(q\)-bic hypersurface associated with a \(q\)-bic form
\((V,\beta)\). Then its Fano scheme \(\FF\) of \(r\)-planes is the
vanishing locus in \(\GG\) of the \(q\)-bic form
\[
\beta_{\mathcal{S}} \colon
\mathcal{S}^{[1]} \otimes_{\sO_{\GG}} \mathcal{S} \to
\sO_{\GG}
\]
obtained by restricting \(\beta\) to \(\mathcal{S}\). In particular,
\(\dim\FF \geq (r+1)(n-2r-1)\) whenever it is nonempty.
\end{Lemma}

\begin{proof}
This follows from the moduli interpretation of the Grassmannian together
the discussion of \parref{basics-fano-schemes}. The dimension estimate is
then:
\[
\dim\FF
\geq
\dim\GG -
\rank_{\sO_{\GG}}(\mathcal{S}^{[1]} \otimes_{\sO_{\GG}} \mathcal{S})
= (r+1)(n-r) - (r+1)^2
= (r+1)(n-2r-1).
\qedhere
\]
\end{proof}

Perhaps the most striking feature of this estimate is that it does not depend
on \(q\), or equivalently, the degree of \(X\). For comparison, the
Fano scheme of \(r\)-planes of a general hypersurface \(Y \subset \PP V\) of
degree \(d\) has codimension \(\binom{d+r}{r}\), see
\cite[Th\'eor\`eme 2.1]{DM:Fano} for example. Even when \(d = q+1\) is small,
the estimate above gives a larger Fano scheme than usually expected. Thus
\parref{basics-fano-equations} might be read as a systematic explanation to the
old observation that \(q\)-bic hypersurfaces contain many more linear subspaces
than usual; compare with \cite[Example 1.27]{Collino} and
\cite[pp.51--52]{Debarre:HDAG}.

\medskip
Consider the parameter space of \(q\)-bic hypersurfaces in \(\PP V\):
\[
\qbics_{\PP V}
\coloneqq \Set{[X] | X \subset \PP V\;\text{a \(q\)-bic hypersurface}}
\coloneqq \PP(V^{[1]} \otimes_\kk V)^\vee.
\]
The following shows that the Fano schemes are nonempty in a certain range:

\begin{Lemma}\label{hypersurfaces-nonempty-fano}
If \(0 \leq r < \frac{n}{2}\), then the Fano scheme \(\FF\) is nonempty.
\end{Lemma}

\begin{proof}
This is a geometric question, so assume that \(\kk\) is algebraically closed.
Consider the incidence correspondence
\[
\mathbf{Inc} \coloneqq
\Set{([U], [X])
\in \GG \times \qbics_{\PP V} |
\PP U \subseteq X}.
\]
The task is to show that the second projection
\(\pr_2 \colon \mathbf{Inc} \to \qbics_{\PP V}\) is surjective, since its fibre
over a point \([X]\) is the Fano scheme of \(X\). Since \(\mathbf{Inc}\) is a
projective space bundle over \(\Gr\) via \(\pr_1\), it is proper and so it
suffices to show that \(\pr_2\) is dominant.
This is the case since the general member of \(\qbics_{\PP V}\) corresponds to a
\(q\)-bic hypersurface whose underlying \(q\)-bic form is nonsingular by
\cite[\href{https://arxiv.org/pdf/2301.09929.pdf\#section.6}{\S\textbf{6}}]{qbic-forms},
and a nonsingular \(q\)-bic form \((V,\beta)\) of dimension \(n+1\) contains an
isotropic subspace of dimension \(\lfloor \frac{n+1}{2} \rfloor\), for
instance, by explicit construction upon choosing a basis of \(V\) in which
\(\beta\) is simply the identity matrix, possible by \citeForms{2.7}.
\end{proof}

\subsectiondash{Numerical invariants}\label{basics-numerics}
It follows from \parref{basics-fano-equations} that, when \(\FF\) is of the
expected codimension \((r+1)^2\) in \(\GG\), its class in the Chow ring of the
Grassmannian is given by the top Chern class
\[
[\FF] = c_{(r+1)^2}\big((\mathcal{S}^{[1]} \otimes_{\sO_{\GG}} \mathcal{S})^\vee\big)
\in \mathrm{CH}^{(r+1)^2}(\GG).
\]
Numerical invariants of \(\FF\) may then by computed via Schubert calculus, as
in \cite[p.271]{Fulton}, at least for small \(r\). For instance, adapting the
argument of \cite[Th\'eor\`eme 4.3]{DM:Fano} gives:

\begin{Proposition}\label{basics-numerics-polynomial}
If \(\FF\) is of expected dimension \((r+1)(n-2r-1)\), then its degree in
its Pl\"ucker embedding is the coefficient of \(x_0^n x_1^n \cdots x_r^{n-r}\)
in
\[
\pushQED{\qed}
(x_0 + x_1 + \cdots + x_r)^{(r+1)(n-2r-1)}  \cdot
\Big(\prod\nolimits_{i,j = 0}^r (x_i + q x_j)\Big) \cdot
\Big(\prod\nolimits_{0 \leq i < j \leq r} (x_i - x_j)\Big).
\qedhere
\popQED
\]
\end{Proposition}

In general, this is rather unwieldy, but there are at least three cases beyond
the \(r = 0\) case for which the degree \(\deg(\FF)\) of the Fano scheme with respect to its
Pl\"ucker line bundle \(\sO_\FF(1)\) has a neat expression. This is summarized
in the following; proofs of the latter two statements are given later:

\begin{Proposition}
If the scheme \(\FF_1\) of lines in \(X\) is of dimension \(2n-6\), then
\[
\deg\FF_1 =
\frac{(2n-6)!}{(n-1)!(n-3)!} (q+1)^2\big((n-1)q^2 + (2n-8)q + (n-1)\big).
\]
For the scheme \(\FF_m\) of half-dimensional planes in \(X\),
\[
\deg\FF_m =
\begin{dcases*}
\prod\nolimits_{i = 0}^r (q^{2i+1}+1) & if \(n = 2m+1\) and \(\dim\FF_m = 0\), and \\
\prod\nolimits_{i = 0}^r \frac{q^{2i+2} - 1}{q-1} & if \(n = 2m+2\) and \(\dim\FF_m = m+1\).
\end{dcases*}
\]
\end{Proposition}

\begin{proof}
The case of lines can be obtained directly from
\parref{basics-numerics-polynomial}. The half-dimensional case follows in the
smooth case via Theorem \parref{theorem-fano-schemes} together with
\parref{hermitian-maximal-count} when \(n = 2m+1\), and
\parref{cgaj-plucker-degree} when \(n = 2m+2\); the general case then follows
by smoothing the \(q\)-bic hypersurface, taking Fano schemes in families, and
using the fact that degrees are invariant in flat families.
\end{proof}

\section{Smoothness and connectedness}\label{section-smoothness-and-connectedness}
This Section is concerned with smoothness and connectedness properties of the
Fano scheme \(\FF\) of \(r\)-planes of a \(q\)-bic hypersurface \(X\)
associated with a \(q\)-bic form \((V,\beta)\) over \(\kk\). The tangent sheaf
of \(\FF\) is described in \parref{differential-identify} and smooth points of
the expected dimension are given a moduli-theoretic description in
\parref{differential-smooth-points}; see also
\parref{differential-singular-locus} for a geometric interpretation. Finally,
\parref{hypersurfaces-fano-connected} shows that \(\FF\) is connected whenever
\(n \geq 2r+2\) by considering the relative Fano scheme of the universal family
of \(q\)-bic hypersurfaces.

In the following, let \(\mathcal{S}\) and \(\mathcal{Q}\) be the tautological
sub- and quotient bundles of the Grassmannian \(\GG\), and \(\mathcal{S}_\FF\)
and \(\mathcal{Q}_\FF\) their restrictions to \(\FF\). Write
\(V_R \coloneqq V \otimes_\kk R\) for any \(\kk\)-algebra \(R\), and let
\(\beta_R\) be the \(q\)-bic form over \(R\) obtained by extension of scalars.
Write \(\kk[\epsilon]\) for the ring of dual numbers over \(\kk\).

\subsectiondash{Tangent sheaves}\label{differential-conormal}
Describe the tangent sheaf \(\mathcal{T}_\FF\) of the Fano scheme by
considering first-order deformations of the underlying moduli problem
\parref{basics-fano-schemes} as follows: First, since the ideal of \(\FF\) in
the Grassmannian \(\GG\) is the image of
\(\beta_{\mathcal{S}} \colon \mathcal{S}^{[1]} \otimes_{\sO_{\GG}} \mathcal{S} \to \sO_{\GG}\)
by \parref{basics-fano-equations}, dualizing the conormal sequence induces an
exact sequence of sheaves of \(\sO_\FF\)-modules given by
\begin{equation*}\label{differential-tangent-sequence}
\tag{\(\star\)}
0 \to
\mathcal{T}_\FF \to
\mathcal{T}_{\GG}\rvert_{\FF} \to
(\mathcal{S}_{\FF}^{[1]} \otimes_{\sO_\FF} \mathcal{S}_{\FF})^\vee.
\end{equation*}
Second, the moduli description of \(\GG\) yields an isomorphism
\(\HomSheaf_{\sO_\GG}(\mathcal{S},\mathcal{Q}) \to \mathcal{T}_\GG\), where
at a point with residue field \(\kappa\) corresponding to an
\((r+1)\)-dimensional subspace \(U\) of \(V_\kappa\), a \(\kappa\)-linear map
\(\varphi \colon U \to V_\kappa/U\) is identified with the
\(\kappa[\epsilon]\)-submodule \(\tilde{U}\) of \(V_{\kappa[\epsilon]}\)
generated by \(\set{u + \epsilon \tilde{\varphi}(u) | u \in U}\), where
\(\tilde{\varphi}(u)\) is any lift of \(\varphi(u)\) to \(V_{\kappa}\).
Third, note that the map
\(
V_\FF \to
V^{[1],\vee}_\FF \twoheadrightarrow
\mathcal{S}_\FF^{[1],\vee}
\)
obtained by composing the adjoint to \(\beta\) with the restriction vanishes on
\(\mathcal{S}_{\FF}\), whence it descends to the quotient to yield a map
\(\nu \colon \mathcal{Q}_\FF \to \mathcal{S}_\FF^{[1],\vee}\). This fits
into an exact sequence of \(\sO_\FF\)-modules given by
\begin{equation*}\label{differential-globalize}
\tag{\(\star\star\)}
0 \to
\mathcal{S}^{[1],\perp}_\FF/\mathcal{S}_\FF \to
\mathcal{Q}_\FF \xrightarrow{\nu}
\mathcal{S}^{[1],\vee}_\FF
\end{equation*}
where \(\mathcal{S}^{[1],\perp}_\FF \subseteq V_\FF\) is the orthogonal with
respect to \(\beta\), as in \parref{basics-qbic-forms}.

Put together, these remarks give the following description of
\(\mathcal{T}_\FF\):

\begin{Proposition}\label{differential-identify}
The first-order deformation \(\tilde{U}\) is isotropic for
\(\beta_{\kappa[\epsilon]}\) if and only if the corresponding linear map
\(\varphi \colon U \to V/U\) factors through \(U^{[1],\perp}/U\). Thus there is
a canonical isomorphism
\[
\mathcal{T}_\FF \otimes_\kk \kappa([\PP U])
\cong \Hom_\kappa(U,U^{[1],\perp}/U),
\]
and the tangent sequence \eqref{differential-tangent-sequence} is
canonically identified with the exact sequence
\(\HomSheaf_{\sO_{\FF}}(\mathcal{S}_{\FF}, \eqref{differential-globalize})\):
\[
0 \to
\HomSheaf_{\sO_{\FF}}(\mathcal{S}_{\FF},\mathcal{S}_{\FF}^{[1],\perp}/\mathcal{S}_{\FF}) \to
\HomSheaf_{\sO_{\FF}}(\mathcal{S}_{\FF},\mathcal{Q}_{\FF}) \xrightarrow{\nu_*}
\HomSheaf_{\sO_{\FF}}(\mathcal{S}_{\FF},\mathcal{S}_{\FF}^{[1],\vee}).
\]
In particular, this canonically identifies \(\mathcal{T}_\FF \cong
\HomSheaf_{\sO_{\FF}}(\mathcal{S}_{\FF},\mathcal{S}_{\FF}^{[1],\perp}/\mathcal{S}_{\FF})\).
\end{Proposition}

\begin{proof}
Since \(\tilde{U}\) is generated by elements \(u + \epsilon \tilde\varphi(u)\)
as \(u\) ranges over \(U\), it is isotropic for \(\beta_{\kappa[\epsilon]}\) if
and only if for every \(u,u' \in U\),
\[
0
= \beta_{\kappa[\epsilon]}\big((u + \epsilon\tilde\varphi(u))^{[1]}, u' + \epsilon\tilde\varphi(u')\big)
= \epsilon\beta_{\kappa[\epsilon]}\big(u^{[1]}, \tilde\varphi(u')\big).
\]
Since \(U^{[1]}\) is spanned by elements of the form \(u^{[1]}\) for
\(u \in U\), this is equivalent to the vanishing of the linear map
\(\nu \circ \varphi \colon U \to U^{[1],\vee}/U\). Thus \(\varphi\) factors
through \(U^{[1],\perp}/U\), yielding the first statement. Comparing with the
descriptions in \parref{differential-conormal}, this identifies the fibre at
\(U\) of the tangent sequence \eqref{differential-tangent-sequence} with that
of the sequence in the statement. Varying over points of \(\FF\) now gives the
second statement.
\end{proof}

Comparing the statement and proof of \parref{differential-identify} with the
sequence \eqref{differential-globalize} gives:

\begin{Corollary}\label{differential-smooth-points}
The following statements are equivalent for the point \([\PP U]\) of \(\FF\):
\begin{enumerate}
\item\label{differential-smooth-points.x}
\([\PP U]\) is a smooth point of \(\mathbf{F}\) of expected
dimension \((r+1)(n-2r-1)\);
\item\label{differential-smooth-points.surjective}
\(\nu \colon \mathcal{Q}_{\mathbf{F}} \to \mathcal{S}^{[1],\vee}_{\mathbf{F}}\)
is surjective on the fibre at \([\PP U]\);
\item\label{differential-smooth-points.injective}
\(\nu^\vee \colon \mathcal{S}^{[1]}_{\mathbf{F}} \to \mathcal{Q}_{\mathbf{F}}^\vee\)
is injective on the fibre at \([\PP U]\); and
\item\label{differential-smooth-points.disjoint}
the subspaces \(U^{[1]}\) and \(V^\perp_\kappa\) of \(V^{[1]}_\kappa\) are
linearly disjoint, that is
\(U^{[1]} \cap V^\perp_{\kappa} = \{0\}\). \qed
\end{enumerate}
\end{Corollary}

Applying the equivalence
\parref{differential-smooth-points}\ref{differential-smooth-points.x}
\(\Leftrightarrow\)
\parref{differential-smooth-points}\ref{differential-smooth-points.injective}
to each generic point of \(\FF\) gives a neat characterization of when the Fano
scheme is generically smooth of the expected dimension:

\begin{Corollary}\label{differential-generic-smoothness}
The following statements regarding the Fano scheme \(\mathbf{F}\) are equivalent:
\begin{enumerate}
\item \(\mathbf{F}\) is generically smooth of the expected dimension
\((r+1)(n-2r-1)\); and
\item the map \(\nu^\vee \colon \mathcal{S}^{[1]}_{\FF} \to \mathcal{Q}_{\FF}^\vee\)
of locally free \(\sO_\FF\)-modules is injective.
\end{enumerate}
In this case, \(\mathbf{F}\) is a local complete intersection scheme, its
tangent sheaf sequence \eqref{differential-tangent-sequence} is exact on the
right, and its dualizing sheaf is a power of its Pl\"ucker line bundle given by
\[
\omega_{\FF}
\cong \sO_{\FF}\big((q+1)(r+1) - (n+1)\big).
\]
\end{Corollary}

\begin{proof}
It remains to compute \(\omega_\FF\). Taking determinants of the cotangent
sequence dual to \eqref{differential-tangent-sequence} gives:
\[
\omega_{\FF} \cong
\det(\Omega^1_{\FF}) \cong
\det(\mathcal{S}^{[1]}_\FF \otimes_{\sO_\FF} \mathcal{S}_\FF)^\vee  \otimes_{\sO_\FF}
\det(\Omega^1_{\GG}\rvert_\FF)
\cong
\sO_\FF\big((q+1)(r+1) - (n+1)\big).
\qedhere
\]
\end{proof}

\subsectiondash{Nonsmooth locus}\label{differential-nonsmooth}
Define the \emph{nonsmooth locus} of \(\FF\) as the closed subscheme given by
the Fitting ideal of its sheaf of differentials in degree \((r+1)(n-2r-1)\),
the expected dimension of \(\FF\):
\[
\Sing\FF \coloneqq
\Fitt_{(r+1)(n-2r-1)}(\Omega^1_{\FF/\mathbf{k}}).
\]
Thus the nonsmooth locus is supported on the points which are either
of larger than expected dimension, or nonsmooth of the expected dimension, see
\citeSP{0C3H}. Dualizing the sequence in \parref{differential-identify}
gives a presentation of \(\Omega^1_{\FF/\kk}\), identifying the Fitting
subscheme as the top degeneracy locus of \(\nu_*^\vee\):
\[
\Sing\FF
=
\mathrm{Degen}_{(r+1)^2 - 1}\big(\nu_*^\vee \colon
  \mathcal{S}^{[1]}_{\FF} \otimes_{\sO_{\FF}} \mathcal{S}_{\FF} \to
  \mathcal{Q}_{\FF}^\vee \otimes_{\sO_{\FF}} \mathcal{S}_{\FF}
\big).
\]
Comparing the construction of \(\nu\) from \parref{differential-conormal}
with
\parref{differential-smooth-points}\ref{differential-smooth-points.disjoint}
shows that \(\nu_*^\vee\) drops rank precisely at points where
\(\mathcal{S}^{[1]}_{\FF}\) intersects \(V_{\FF}^\perp\), and the Cauchy--Binet
formula implies that this holds scheme-theoretically:
\[
\Sing\FF =
\mathrm{Degen}_{(r+1)^2-1}\big(
  \mathcal{S}^{[1]}_{\FF} \otimes_{\sO_{\FF}} \mathcal{S}_{\FF} \to
  (V^{[1]}/V^\perp)_{\FF} \otimes_{\sO_{\FF}} \mathcal{S}_\FF
\big).
\]
When \(r = 0\), so that \(\FF = X\) is simply the hypersurface, this yields a
clean moduli theoretic description of the singular locus of \(X\), and
furthermore gives a criterion for smoothness over \(\kk\):

\begin{Corollary}\label{differential-smoothness-X}
The nonsmooth locus of a \(q\)-bic hypersurface \(X\) is the closed subscheme
\[
\Sing X =
\mathrm{V}\big(
  \sO_X(-q) \xrightarrow{\mathrm{eu}^{[1]}}
  V^{[1]}_X \twoheadrightarrow
  (V^{[1]}/V^\perp)_X
\big)
\]
corresponding to the linear space \(\PP V^\perp\) in \(\PP V^{[1]}\). In
particular, the following are equivalent:
\begin{enumerate}
\item\label{differential-smoothness-X.smooth}
\(X\) is smooth over \(\kk\);
\item\label{differential-smoothness-X.beta}
\(\beta\) is nonsingular; and
\item\label{differential-smoothness-X.Gram}
\(\Gram(\beta;e_0,\ldots,e_n)\) is invertible for any basis \(V = \langle e_0,\ldots,e_n \rangle\).
\end{enumerate}
Moreover, \(X\) is regular if and only if \(V^{\perp}\) has trivial
Frobenius descent to \(V\) over \(\kk\).
\end{Corollary}

\begin{proof}
The first part specializes the conclusion from
\parref{differential-nonsmooth} to the case \(r = 0\). This implies the
third part and the equivalence between
\ref{differential-smoothness-X.smooth} and
\ref{differential-smoothness-X.beta} since \(V^\perp\) is always isotropic for
\(\beta\). Finally, the equivalence of \ref{differential-smoothness-X.beta} and
\ref{differential-smoothness-X.Gram} is straightforward, see \citeForms{1.2}.
\end{proof}

This gives a geometric interpretation of the characterization
\parref{differential-smooth-points}\ref{differential-smooth-points.x}:

\begin{Corollary}\label{differential-singular-locus}
A point of \(\FF\) is smooth of expected dimension if and only if the
corresponding \(r\)-plane is disjoint from the nonsmooth locus of \(X\). In
other words,
\[
\Sing\mathbf{F} =
\Set{[\PP U] \in \mathbf{F} | \PP U \cap \Sing X \neq \varnothing}.
\]
In particular, \(\FF\) is smooth of the expected dimension if and only if \(X\)
is. \qed
\end{Corollary}

Generally, \(\Sing\FF\) is complicated. But at least when
\(X\) has an isolated singularity---equivalently by
\parref{differential-smoothness-X}, is of corank \(1\)---and is geometrically
not a cone, its dimension can be determined as follows:

\begin{Lemma}\label{hypersurfaces-fano-corank-1}
Assume \(X\) is of corank \(1\) and geometrically not a cone.
If \(n \geq 2r+1\), then
\[ \dim \Sing \FF = r(n-2r-1). \]
In particular, \(\FF\) is of expected dimension if \(n \geq 2r+1\), whence
generically smooth if \(n \geq 2r+2\).
\end{Lemma}

\begin{proof}
This is a geometric statement, so assume that \(\kk\) is algebraically closed.
Let \(L \subset V\) be the \(1\)-dimensional subspace descending
\(V^\perp \subset V^{[1]}\) that underlies the unique singular point
\(x \in X\) by \parref{differential-smoothness-X}. Proceed by induction on
\(r \geq 0\), wherein the base case \(r = 0\) follows from the fact that \(X\)
has the unique singular point \(x\). Assume \(r \geq 1\). The facts from
\parref{basics-classification} show that \(X\) is not a cone if and only if
\(L \neq V^{[1],\perp}\). Therefore \(\beta(-,L)\) is nonzero on \(V^{[1]}\),
and so the linear space
\[
L^\perp
\coloneqq \ker(\beta^\vee \colon V^{[1]} \to V^\vee \twoheadrightarrow L^\vee)
\]
is of codimension \(1\) in \(V^{[1]}\). Its Frobenius descent \(L^\dagger\) to
\(V\) contains \(L\), and so the restriction of \(\beta\) thereon is of corank
at most \(2\) and has radical \(L\); thus there is an induced \(q\)-bic form
\(\bar{\beta}\) of corank at most \(1\) and no radical on \(L^\dagger/L\);
in other words, \(X \cap \PP L^\dagger\) is a cone with vertex \(x\) over the
non-conical \(q\)-bic \((n-3)\)-fold \(\bar{X}\) of corank at most \(1\) in
\(\PP(L^\dagger/L)\). Since any \(r\)-plane in \(X\) through \(x\) is a cone
over an \((r-1)\)-plane in \(\bar{X}\), this gives the first equality in
\[
\dim \Sing \FF =
\dim\FF_{r-1}(\bar{X}) =
r(n-2r-1).
\]
If the corank of \(\bar{X}\) is \(0\), then it is smooth by
\parref{differential-smoothness-X} and \(\FF_{r-1}(\bar{X})\) has its expected
dimension by \parref{differential-generic-smoothness}. If the corank of
\(\bar{X}\) is \(1\), then since \(n-2 \geq 2(r-1) + 1\) is equivalent to the
inequality in the hypothesis of the Lemma, induction applies to
\(\FF_{r-1}(\bar{X})\) to show it is of expected dimension. In either case,
this gives the second equality above, and whence the result.
\end{proof}

Adapting the argument of \cite[Theorem 6]{BVV:Fano}---see also
\cite[p.544]{DM:Fano}---shows that the Fano schemes are connected whenever they
are positive dimensional:

\begin{Proposition}\label{hypersurfaces-fano-connected}
If \(n \geq 2r+2\), then \(\FF\) is connected.
\end{Proposition}

\begin{proof}
This is a geometric question, so assume that \(\kk\) is algebraically closed.
Consider once again, as in the proof of \parref{hypersurfaces-nonempty-fano},
the incidence correspondence
\[
\mathbf{Inc} \coloneqq
\Set{([U], [X]) \in \mathbf{G} \times \qbics_{\PP V} |
\PP U \subseteq X}.
\]
The locus \(Z \subset \mathbf{Inc}\) over which \(\pr_2 \colon \mathbf{Inc} \to \qbics_{\PP V}\) is
not smooth has codimension at least \(n-2r\):
Indeed, \(\pr_2\) is smooth above the open subset of \(\qbics_{\PP V}\)
parameterizing smooth \(q\)-bic hypersurfaces by
\parref{differential-singular-locus}. A general point of its codimension \(1\)
complement corresponds to a \(q\)-bic hypersurface \(X\) of corank \(1\) which
is not a cone by
\cite[\S\href{https://chngr.github.io/assets/qbic-forms.pdf\#section.6}{\textbf{6}}]{qbic-forms}.
Over such a point, \(\FF = \pr_2^{-1}([X])\) is of expected dimension \((r+1)(n-2r-1)\)
and \(\Sing\FF\) has dimension \(r(n-2r-1)\) by
\parref{hypersurfaces-fano-corank-1}, so
\[
\codim(Z \subseteq \mathbf{Inc}) \geq 1 +
\codim(\Sing\FF \subseteq \FF) = n-2r.
\]

Let
\(\pr_2 \colon \mathbf{Inc} \to \mathbf{Inc}' \to \qbics_{\PP V}\)
be the Stein factorization of the second projection. Then \(Z\) contains the
preimage of the branch locus of \(\mathbf{Inc}' \to \qbics_{\PP V}\). Since
\(n \geq 2r+2\), the codimension estimate implies that the branch locus has
codimension at least \(2\). Purity of the Branch Locus, as in
\citeSP{0BMB}, then implies that \(\mathbf{Inc}' \to \qbics_{\PP V}\) is
\'etale, and hence an isomorphism since projective space is simply connected.
Therefore the fibres of \(\pr_2\) are connected, meaning that each of the Fano
schemes \(\FF\) is connected.
\end{proof}

Finally, general considerations also show that the Fano schemes are
often simply connected, in the sense that there are no nontrivial connected
\'etale covers:

\begin{Proposition}\label{hypersurfaces-simply-connected}
If \(n \geq 3r+3\) and \(\FF\) is of expected dimension, then \(\FF\) is
simply connected.
\end{Proposition}

\begin{proof}
This follows from \cite[Corollaire 7.4]{Debarre:Connectivity}, which says that
\(\FF\), being the zero locus of a vector bundle in a Grassmannian, is simply
connected if the top Chern class of the vector bundle
\[
(\mathcal{S}^{[1]} \otimes_{\sO_{\mathbf{G}}} \mathcal{S})^\vee \oplus
(\mathcal{S}^{[1]} \otimes_{\sO_{\mathbf{G}}} \mathcal{S})^\vee \oplus
\mathcal{S}^\vee
\]
of rank \((r+1)(2r+3)\) on \(\GG\) is nonzero, and this happens when
\(n \geq 3r+3\).
\end{proof}

As is noted in \cite[Exemple 3.3]{DM:Fano}, this is not optimal: for instance,
this does not apply to the fourfold of lines on a cubic fourfold. Note,
however, that \parref{hypersurfaces-simply-connected} does apply to all the
Fano schemes of \(q\)-bic hypersurfaces with trivial canonical bundle so
as long as \(q > 2\).

\section{Fano correspondences}\label{section-correspondences}
Let \(0 \leq k < r < n\). This Section is concerned with the varieties
\[
\mathbf{F}_{r,k} \coloneqq
\mathbf{F}_{r,k}(X) \coloneqq
\Set{(\PP W \supset \PP U) | X \supseteq \PP W \supset \PP U, \PP W \cong \PP^r, \PP U \cong \PP^k}
\]
parameterizing nested pairs of \(k\)- and \(r\)-planes in \(X\). The projections
\[
\begin{tikzcd}
& \mathbf{F}_{r,k} \ar[dl,"\pr_r"'] \ar[dr,"\pr_k"] \\
\FF_r && \FF_k
\end{tikzcd}
\]
exhibit this as a correspondence between the Fano schemes of \(k\)- and \(r\)-planes
in \(X\); any such incidence correspondence is referred to as a \emph{Fano
correspondence}. The projection \(\pr_r \colon \FF_{r,k} \to \FF_r\) identifies
the Fano correspondence as the Grassmannian bundle
\(\mathbf{G}(k+1,\mathcal{S}_{\FF_r})\) on the tautological subbundle of
\(\FF_r\). Therefore
\[
\dim\FF_{r,k} \geq
(r+1)(n-2r-1) + (k+1)(r-k)
\]
with equality if and only if the Fano scheme \(\FF_r\) itself has its expected
dimension \((r+1)(n-2r-1)\).

The fibres of \(\pr_k \colon \FF_{r,k} \to \FF_k\), up to nilpotents, look like
Fano schemes of lower dimensional \(q\)-bic hypersurfaces. To give a precise
statement, consider a point of \(\FF_k\) with residue field \(\kappa\),
corresponding to an \((k+1)\)-dimensional isotropic subspace
\(U \subseteq V_\kappa\). Let
\[
U^\dagger \coloneqq U^{[1],\perp} \cap U^{\perp,[-1]}
\]
be the intersection of the two orthogonals of \(U\). The restriction of
\(\beta\) to \(U^\dagger\) has \(U\) in its radical, so there is an induced
\(q\)-bic form on the quotient \(U^\dagger/U\); let \(\bar{X}\) be the
associated \(q\)-bic hypersurface.

\begin{Lemma}\label{incidences-containing-a-plane}
The reduced scheme parameterizing \(r\)-planes in \(X\) containing a fixed
\(k\)-plane is canonically isomorphic to the reduced Fano scheme of \((r-k-1)\)-planes
in the \(q\)-bic hypersurface \(\bar{X}\):
\[
\pr^{-1}_k([\PP U])_{\mathrm{red}} \cong
\mathbf{F}_{r-k-1}(\bar{X})_{\mathrm{red}}.
\]
\end{Lemma}

\begin{proof}
This is a geometric statement, so assume that \(\kappa = \kk\).
The fibre \(\pr^{-1}_k([\PP U])\) projects isomorphically along \(\pr_r\) onto
the scheme of \(r\)-planes in \(X\) containing \(\PP U\). In other words,
setting \(\mathbf{G} \coloneqq \mathbf{G}(r+1,V)\), this gives an
identification
\[
\pr_k^{-1}([\PP U]) \cong
\Set{[W] \in \mathbf{G} | U \subseteq W \;\text{isotropic for}\;\beta} =
\FF_r \cap \bar{\mathbf{G}},
\]
where \(\bar{\mathbf{G}}\) is the subvariety of the Grassmannian
parameterizing subspaces containing \(U\). Proceed by describing \(\pr_k^{-1}([\PP U])\)
as a subscheme of \(\bar{\mathbf{G}}\) as follows:

Taking images and preimages along the quotient map \(V \to V/U\) provides an
isomorphism between \(\bar{\mathbf{G}}\) and the Grassmannian
\(\mathbf{G}(r-k,V/U)\). Let \(\bar{\mathcal{S}}\) be the tautological
subbundle of rank \(r-k\) in \(V/U\) over \(\bar{\mathbf{G}}\). A choice of
splitting \(V \cong U \oplus V/U\) induces a splitting
\(\mathcal{S}\rvert_{\bar{\mathbf{G}}} \cong U_{\bar{\mathbf{G}}} \oplus \bar{\mathcal{S}}\)
of the tautological subbundle \(\mathcal{S}\) of \(\mathbf{G}\) restricted to
\(\bar{\mathbf{G}}\). The restriction of \(\beta\) to
\(\mathcal{S}\rvert_{\bar{\mathbf{G}}}\) now splits into four components:
the---vanishing!---restriction to \(U\), and the three maps
\[
\bar{\beta} \colon
\bar{\mathcal{S}}^{[1]}
  \otimes_{\sO_{\bar{\mathbf{G}}}} \bar{\mathcal{S}}
  \to \sO_{\bar{\mathbf{G}}},
\qquad
\bar{\beta}_1 \colon
U_{\bar{\mathbf{G}}}^{[1]}
  \otimes_{\sO_{\bar{\mathbf{G}}}} \bar{\mathcal{S}}
  \to \sO_{\bar{\mathbf{G}}},
\qquad
\bar{\beta}_2 \colon
\bar{\mathcal{S}}^{[1]}
  \otimes_{\sO_{\bar{\mathbf{G}}}} U_{\bar{\mathbf{G}}}
  \to \sO_{\bar{\mathbf{G}}}.
\]
Then \(\pr^{-1}_k([\PP U])\) is defined in \(\bar{\mathbf{G}}\) by the
vanishing of these maps. Vanishing of \(\bar{\beta}_1\) and \(\bar{\beta}_2\)
mean that
\[
\bar{\mathcal{S}} \subseteq U_{\bar{\mathbf{G}}}^{[1],\perp}
\qquad\text{and}\qquad
\bar{\mathcal{S}}^{[1]} \subseteq U_{\bar{\mathbf{G}}}^\perp
\]
respectively, whence any geometric point of \(\pr^{-1}_k([\PP U])\) lies in the
subvariety of \(\bar{\mathbf{G}}\) parameterizing subspaces contained in
\(U^\dagger/U\). Vanishing of \(\bar\beta\) means that the corresponding
subspace is isotropic for the \(q\)-bic form induced by \(\beta\) thereon, as
required.
\end{proof}

Since \(U\) is of dimension \(k+1\) in the \((n+1)\)-dimensional
space \(V\), the quotient \(U^\dagger/U\) has dimension at least \(n-3k-2\).
Combined with the dimension estimate of \parref{basics-fano-equations}, the
identification of \parref{incidences-containing-a-plane} implies:

\begin{Corollary}\label{incidences-fibre-dimension-estimate}
The morphism \(\pr_k \colon \FF_{r,k} \to \FF_k\) is surjective whenever
\(n \geq 2r+k+2\). \qed
\end{Corollary}

This implies that \(\FF_k\) is covered by Grassmannians, whence
uniruled, when \(n \geq 2r+k+2\). Taking \(k = 0\), this implies that a
\(q\)-bic hypersurface is covered in \(r\)-planes when \(n \geq 2r+2\).

When the form is nonsingular, the analysis of
\parref{incidences-containing-a-plane} gives more:

\begin{Corollary}\label{basics-incidences-nondegenerate}
Assume moreover that \(\beta\) is nonsingular. Then, for every closed point
\([\PP U] \in \FF_k\),
\[
\dim\pr_k^{-1}([\PP U]) \leq (r-k)(n-2r-1),
\]
with equality if and only if \(U\) is Hermitian, in which case
\(\pr_k^{-1}([\PP U])\) is furthermore smooth.
\end{Corollary}

\begin{proof}
This is a geometric statement, so assume the residue field of \(U\)
is \(\kk\). Then
\(\dim_\kk U^\dagger/U \leq n-2k-1\) with equality if and only if the two
orthogonals of \(U\) coincide which, by \parref{hermitian-subspaces}, is
equivalent to \(U\) being Hermitian. This gives the statements regarding
dimension.

It remains to see that the fibre is smooth when \(U\) is Hermitian. Consider
again its equations given in the proof of
\parref{incidences-containing-a-plane}: First, that \(U\) is Hermitian means
that the equations \(\bar{\beta}_2\) are \(q\)-powers of the linear equations
\(\bar{\beta}_1\), and are thus redundant. Second, the \(q\)-bic form
\(\bar\beta\) induced by \(\beta\) on \(U^\dagger/U\) is nonsingular: Since
\(\beta\) is nonsingular, \(U^\dagger\) has codimension \(k+1\) in \(V\) and
the restriction of \(\beta\) thereon has corank at most \(k+1\). Since \(U\) is
contained in the radical of \(\beta\) restricted to \(U^\dagger\), it follows
that \(\bar\beta\) is nonsingular. Therefore \(\pr^{-1}_k([\PP U])\) is
isomorphic to the Fano scheme of a smooth \(q\)-bic hypersurface \(\bar{X}\) in
the Grassmannian \(\mathbf{G}(r-k,U^\dagger/U)\), and is thus smooth by
\parref{differential-singular-locus}.
\end{proof}

\subsection{Action of the Fano correspondence}\label{cgaj-action}
Suppose now that \(X\) is smooth, and consider the Fano correspondence
\(\mathbf{L} \coloneqq \mathbf{F}_{r,0}\) going between \(X\) and its
scheme \(\FF\) of \(r\)-planes. The scheme \(\mathbf{L}\) is
naturally a closed subscheme of \(\FF \times X\), and may therefore be viewed
as a correspondence of degree \(-r\) from \(\FF\) to \(X\). This acts on Chow
groups
\[
\mathbf{L}_* \colon \mathrm{CH}_*(\FF) \to \mathrm{CH}_{* + r}(X)
\quad\text{and}\quad
\mathbf{L}^* \colon \mathrm{CH}^*(X) \to \mathrm{CH}^{*-r}(\FF)
\]
via \(\mathbf{L}_*(\alpha) \coloneqq \pr_{X,*}(\pr_{\FF}^*(\alpha) \cdot \mathbf{L})\)
and \(\mathbf{L}^*(\beta) \coloneqq \pr_{\FF,*}(\pr_X^*(\beta) \cdot \mathbf{L})\),
where the dot denotes the intersection product of cycles on \(\FF \times X\).

The next statement shows that \(\mathbf{L}\) relates, in particular, the
first Chern classes
\[
h \coloneqq c_1(\sO_X(1)) \in \mathrm{CH}^1(X)
\quad\text{and}\quad
g \coloneqq c_1(\sO_{\FF}(1)) \in \mathrm{CH}^1(\FF)
\]
of the standard polarization of \(X\) as a hypersurface, and the Pl\"ucker
polarization of \(\FF\):

\begin{Lemma}\label{cgaj-correspondence-polarization}
If \(n \geq 2r+2\), then for every \(r+1 \leq k \leq n-1\),
\[
\mathbf{L}^*(h^k) = c_{k-r}(\mathcal{Q}) \in
\mathrm{CH}^{k-r}(\FF)
\]
and is nonzero. In particular, \(\mathbf{L}^*(h^{r+1}) = g\) in
\(\mathrm{CH}^1(\FF)\).
\end{Lemma}

\begin{proof}
Since \(h^k\) represents a general \((n-k)\)-plane section of \(X\),
\(\mathbf{L}^*(h^k)\) represents the scheme of \(r\)-planes contained in
\(X\) which are incident with a fixed, but general, \((n-k)\)-plane
section; in other words, \(\mathbf{L}^*(h^k)\) represents the restriction to
\(\FF\) of the Schubert variety corresponding to incidence with a general
\((n-k)\)-plane, which by \cite[Example 14.7.3]{Fulton} is given by
\(c_{k-r}(\mathcal{Q})\). Since \(X\) is covered by \(r\)-planes by
\parref{incidences-fibre-dimension-estimate}, it follows that this class is
nonzero on \(\FF\).
\end{proof}

\subsection{Incidence schemes}\label{correspondence-incidence}
Continuing with \(X\) smooth, suppose further that \(n \geq 2r+2\), so that,
by \parref{incidences-fibre-dimension-estimate}, \(X\) is covered by
\(r\)-planes. Fix an \(r\)-plane \(P\) in \(X\) and consider the schemes of
\(k\)-planes incident with \(P\) for varying \(k\). Namely, for each \(0 \leq k
\leq r\), consider the \((k(n-2k-2)+r)\)-cycle
\[
[D_{k,P}] \coloneqq \FF_{k,0}^*([P])
\in \mathrm{CH}^{n-r-k-1}(\FF_k)
\]
obtained by applying the Fano correspondence \(\FF_{k,0}\) between \(\FF_k\)
and \(X\) to the class of \(P\) in \(\mathrm{CH}^{n-r-1}(X)\). The cycle
is supported on the closure of the locus
\[
\Set{[P'] \in \FF_k | P' \neq P \;\text{and}\; P' \cap P \neq \varnothing}
\]
of \(k\)-planes in \(X\) incident with \(P\). The basic property of these cycles
is:

\begin{Lemma}\label{correspondence-plucker}
The \(D_{k,P}\) are algebraically equivalent for varying \(r\)-planes
\(P\), and
\[ (q+1)[D_{k,P}] \sim_{\mathrm{alg}} c_{n-r-k-1}(\mathcal{Q}). \]
\end{Lemma}

\begin{proof}
Since the Fano scheme of \(r\)-planes in \(X\) is connected when
\(n \geq 2r+2\) by \parref{hypersurfaces-fano-connected}, the classes \([P]\)
of \(r\)-planes in \(X\) are algebraically equivalent for varying \(P\), whence
by \cite[Proposition 10.3]{Fulton}, the \([D_{k,P}]\) are also algebraically
equivalent for varying \(P\).

To now show that \((q+1)[D_{k,P}]\) is algebraically equivalent with
\(c_{n-r-k-1}(\mathcal{Q})\), it suffices by
\parref{cgaj-correspondence-polarization} to show that there exists an
\((r+1)\)-plane section of \(X\) supported on an \(r\)-plane. By first taking
general hyperplane sections, it suffices to consider the case \(n = 2r+2\), and
this follows from \parref{incidences-hermitian-perp} below.
\end{proof}

The following shows that the Chow class of a Hermitian plane is, up to
multiplication by \(q+1\), a power of the hyperplane class. Of particular note
is the case of a smooth \(q\)-bic curve \(C\), whereupon this shows that
\(q+1\) times a Hermitian point coincides with the standard planar polarization
on \(C\).

\begin{Lemma}\label{incidences-hermitian-perp}
Let \(X\) be a smooth \(q\)-bic hypersurface of dimension \(2r+1\) containing
a Hermitian \(r\)-plane \(\PP U\). Then the \((r+1)\)-plane section
\(X \cap \PP U^\dagger\) is supported on \(\PP U\) with multiplicity \(q+1\).
\end{Lemma}

\begin{proof}
Indeed, as in the proof of \parref{basics-incidences-nondegenerate}, the
orthogonal \(U^\dagger = U^{[1],\perp}\) is an \((r+2)\)-dimensional subspace
in \(V\) such that the restriction of \(\beta\) thereon is of rank \(1\) with
\(U\) as its radical. This means that, by \parref{basics-hyperplane-section},
the \((r+1)\)-plane section \(X \cap \PP U^\dagger\) is but \(\PP U\) with
multiplicity \(q+1\).
\end{proof}

\section{Hermitian structures}\label{section-hermitian}
This Section is concerned with the geometry associated with the canonical
\(\FF_{q^2}\)-structure on \(V\) provided by a nonsingular \(q\)-bic form
\(\beta\) as in \parref{hermitian-structures}. Namely, let
\(\sigma_\beta \colon V^{[2]} \to V\) and \(\phi \colon V \to V\) be the
\(\FF_{q^2}\)-descent datum and corresponding absolute Frobenius morphism
associated with \(\beta\). Embedding the Frobenius-twisted subbundle
\(\mathcal{S}^{[2]}\) of the Grassmannian \(\mathbf{G} \coloneqq
\mathbf{G}(r+1,V)\) into \(V\) via \(\sigma_\beta\) induces a
\(\kk\)-endomorphism \(\phi_\GG \colon [U] \mapsto [\phi(U)]\). Since \(\phi\)
preserves isotropic vectors for \(\beta\) by
\parref{hermitian-phi-properties}\ref{hermitian-phi-properties.isotropic}, this
restricts to a \(\kk\)-endomorphism
\[
\phi_\FF \colon \FF \to \FF
\]
on the Fano scheme \(\FF\) of \(r\)-planes in the \(q\)-bic hypersurface
\(X\) associated with \(\beta\). For instance, when \(r = 0\), the endomorphism
\(\phi_X \colon X \to X\) can be described geometrically as follows: The
intersection of \(X\) with the embedded tangent space at a point \(x\) is a
corank \(1\) \(q\)-bic hypersurface \(X'\), and \(\phi_X(x)\) is the special
point of \(X'\) corresponding to the right-kernel of the underlying \(q\)-bic
form. For example, when \(X\) is a \(q\)-bic curve, the tangent line at \(x\)
has multiplicity \(q\) at \(x\), and the residual point of intersection is
\(\phi_X(x)\); see \citeThesis{2.9.9} for more details.

\subsectiondash{Dynamical filtration}\label{hermitian-filtration}
Dynamics of \(\phi_\FF\) distinguish special loci in the Fano schemes.
Of particular interest here is the case \(r = 0\), regarding the hypersurface
itself: for each \(0 \leq k \leq n/2\), let
\[
X^k \coloneqq
\Set{[L] \in \PP V |
\phi^{[0,k]}(L) \coloneqq \langle L, \phi(L), \ldots, \phi^k(L)\rangle
\;\text{is isotropic for \(\beta\)}}
\]
be the locus of lines in \(V\) whose \(k\)-th cyclic subspace generated by
\(\phi\) is isotropic; of course, this description can be made functorial as in
\parref{basics-definition}. These fit into a descending sequence
\[
X^{\smallbullet} \colon
X
= X^0
\supset X^1
\supset \cdots
\supset X^{\lfloor n/2 \rfloor}.
\]
Each piece of the filtration is given by a complete intersection:

\begin{Proposition}\label{hermitian-complete-intersection}
The locus \(X^k\) is the complete intersection in \(\PP V\) given by the
vanishing of
\[
\beta(\phi_{\PP V}^{*,i}(\mathrm{eu})^{[1]}, \mathrm{eu}) \colon
\sO_{\PP V} \to
\sO_{\PP V}(q^{2i+1}+1)
\quad\text{for}\;0 \leq i \leq k.
\]
The singular locus of \(X^k\) is supported on the union of the Hermitian
\((k-1)\)-planes in \(X\).
\end{Proposition}

\begin{proof}
The space \(\phi^{[0,k]}(L)\) is isotropic if and only if
\(\beta(\phi^i(L)^{[1]}, \phi^j(L)) = 0\) for each \(0 \leq i, j \leq k\).
The first statement now follows upon successively applying the identities of
\parref{hermitian-phi-properties}, since
\[
\beta(\phi^i(L)^{[1]},\phi^j(L)) =
\begin{dcases*}
\beta(\phi^{i-j}(L)^{[1]},L)^{q^{2j}} & if \(i \geq j\), and \\
\beta(\phi^{j-i-1}(L)^{[1]},L)^{q^{2i+1}} & if \(i < j\).
\end{dcases*}
\]

Computing as in \parref{differential-identify} shows that the tangent space to
\(X^k\) at a point \(x = [L]\) is the vector space
\[
\mathcal{T}_{X^k} \otimes_{\sO_{X^k}} \kappa(x) \cong
\Hom_{\kappa(x)}(L, \phi^{[0,k]}(L)^{[1],\perp}/L).
\]
This has dimension \(n - 1 - \dim_{\kappa(x)} \phi^{[0,k]}(L)\) since
\(\beta\) is nondegenerate. Thus \(x\) is a singular point if and only if
\(\dim_{\kappa(x)}\phi^{[0,k]}(L) \leq k\); by \parref{hermitian-minimal}, this
occurs if and only if \(L\) lies in a Hermitian subspace of dimension \(k\),
yielding the second statement.
\end{proof}

\subsectiondash{}\label{hermitian-coordinates}
The equations of \(X^k\) are particularly simple upon choosing a
basis \(V = \langle e_0,\ldots,e_n\rangle\) consisting of Hermitian vectors.
It follows from its definition in \parref{hermitian-structures} and
\parref{hermitian-phi-properties}\ref{hermitian-phi-properties.fixed} that the
endomorphism \(\phi\) is the \(\kk\)-linear \(q^2\)-power Frobenius in the
corresponding coordinates:
\[
\phi_{\PP^n} \colon \PP^n \to \PP^n
\qquad (x_0:\cdots:x_n) \mapsto (x_0^{q^2}: \cdots: x_n^{q^2}).
\]
In particular, the Hermitian subspaces of \(X\) coincide with its
\(\FF_{q^2}\)-rational ones, and
\[
X^k
= \bigcap\nolimits_{s = 0}^k
\mathrm{V}\Big(\sum\nolimits_{i,j = 0}^n a_{ij} x_i^{q^{2s+1}} x_j\Big)
\subset \PP^n
\]
where \((a_{ij})_{i,j = 0}^n \coloneqq \Gram(\beta;e_0,\ldots,e_n)\) as
in \parref{basics-equations}.
When \(X\) is given by the Fermat equation, this filtration is seen to coincide
with that defined by Lusztig in \cite[Definition 2]{Lusztig:Green}.

\medskip
The final piece of \(X^{\smallbullet}\) is the union of the maximal Hermitian
isotropic subspaces of \(V\):

\begin{Lemma}\label{hermitian-maximal}
The scheme \(X^{\lfloor n/2 \rfloor}\) is supported on the union of the maximal
Hermitian subspaces contained in \(X\). Moreover, the sections
\[
\beta(\phi^{*,i}_{\PP V}(\mathrm{eu})^{[1]}, \mathrm{eu}) \colon
\sO_{\PP V} \to \sO_{\PP V}(q^{2i+1}+1)
\]
vanish on \(X^{\lfloor n/2 \rfloor}\) for all \(i \geq 0\).
\end{Lemma}

\begin{proof}
Writing \(n\) as \(2m+1\) or \(2m+2\), it suffices to show that each
\(x \in X^{\lfloor n/2 \rfloor}\) is contained in a Hermitian \(m\)-plane.
When \(n = 2m+1\), any \(x \in X^m\) lies in some \(m\)-plane
\(\PP U \subset X\) by the description in \parref{hermitian-filtration}; since
\(U\) is isotropic and half-dimensional in \(V\), \(U^{[1],\perp} = U\)
and \(U^\perp = U^{[1]}\), and so \(U\) is Hermitian by
\parref{hermitian-subspaces}\ref{hermitian-subspaces.perps}.
When \(n = 2m+2\), any \(x \in X^{m+1}\) satisfies
\[
\langle x, \phi(x), \ldots, \phi^m(x), \phi^{m+1}(x) \rangle =
\langle x, \phi(x), \ldots, \phi^m(x) \rangle
\]
since any linear subvariety of \(X\) has dimension at most \(m\); therefore
\(x\) lies in a Hermitian \(m\)-plane of \(X\) by \parref{hermitian-minimal}.
The second statement now follows from \parref{hermitian-complete-intersection}
since for every \(v \in V\) contained in an isotropic Hermitian subspace of
dimension \(m+1\) and every \(i > m\),
\(\phi^i(v) = \sum\nolimits_{j = 0}^m a_j \phi^j(v)\) for some \(a_j \in \kk\).
\end{proof}




For example, combined with \parref{hermitian-coordinates}, this implies that
the union of the lines contained in the Fermat \(q\)-bic surface \(X\) is the
complete intersection in \(\PP^3\) given by
\[
x_0^{q+1} + x_1^{q+1} + x_2^{q+1} + x_3^{q+1} =
x_0^{q^3+1} + x_1^{q^3+1} + x_2^{q^3+1} + x_3^{q^3+1} = 0.
\]

Generally, let \(0 \leq k \leq n/2\), and let
\[
\FF_k(X)_{\mathrm{Herm}} \coloneqq
\Set{[\PP U] \in \FF_k(X) | U\;\text{is a Hermitian subspace}}
\]
be the locus of \(k\)-planes in \(X\) whose underlying space is Hermitian.
This is a finite set of points since there are only finitely many Hermitian
vectors by \citeForms{2.5}. Analyzing the schematic structure of
\(X^{\lfloor n/2 \rfloor}\) gives a geometric method to count the number of
maximal isotropic Hermitian subspaces contained in a smooth \(q\)-bic
hypersurface \(X\). The following count is classical: see
\cite[n.32]{Segre:Hermitian} and \cite[Theorem 9.2]{BC:Hermitian}; see also
\cite[Corollary 2.22]{Shimada:Lattices}.

\begin{Corollary}\label{hermitian-maximal-count}
\(\displaystyle
\#\FF_m(X)_{\mathrm{Herm}} =
\begin{dcases*}
\prod\nolimits_{i = 0}^m (q^{2i+1} + 1) & if \(\dim X = 2m\) is even, and \\
\prod\nolimits_{i = 0}^m (q^{2i+3} + 1) & if \(\dim X = 2m+1\) is odd.
\end{dcases*}
\)
\end{Corollary}

\begin{proof}
When \(\dim X = 2m\) is even, \parref{hermitian-maximal} implies that \(X^m\)
is the union of the Hermitian \(m\)-planes in \(X\), and the second statement
of \parref{hermitian-complete-intersection} shows that each plane is generically
reduced. Therefore
\[
\#\FF_m(X)_{\mathrm{Herm}} =
\deg(X^m) =
\prod\nolimits_{i = 0}^m (q^{2i+1}+1).
\]

When \(\dim X = 2m+1\) is odd, \parref{hermitian-maximal} still implies
that \(X^{m+1}\) is the union of the Hermitian \(m\)-planes in \(X\), but
\parref{hermitian-complete-intersection} now shows that each component is
generically nonreduced. Thus it remains to show that each component
\(P \subseteq X^{m+1}\) arises with multiplicity \(q+1\). Choose a
Hermitian basis \(U = \langle u_0,\ldots,u_m \rangle\) of the linear space
underlying \(P\). Since \(\beta\) is nondegenerate, there exists linearly
independent vectors \(v_0,\ldots,v_m \in V\) such that
\[
\beta(u_i^{[1]}, v_j) = \beta(v_j^{[1]}, u_i) = \delta_{ij}.
\]
Complete this to a basis of \(V\) with a Hermitian basis \(w\) of the orthogonal
complement of these vectors. Let \((x_0:\cdots:x_m:y_0:\cdots:y_m:z)\)
be the coordinates of \(\PP V = \PP^{2m+2}\) associated with this basis.
The generic point \(\eta\) of \(P\) is contained in the affine open where
\(x_0 \neq 0\), thus the completed local ring of \(X^{m+1}\) at \(\eta\) is
isomorphic to quotient of the power series ring
\(\kk(x_1,\ldots,x_m)[\!\![ y_0,\ldots,y_m,z ]\!\!]\)
by the ideal generated by the polynomials
\[
z^{q^{2i + 1}+1} + y_0 + y_0^{q^{2i+1}+1} +
\sum\nolimits_{j = 1}^m (x_j^{q^{2i+1}} y_j + x_j y_j^{q^{2i+1}})
\quad\text{for}\;
0 \leq i \leq m+1,
\]
compare with \parref{hermitian-coordinates}.
Solving for \(y_i\) with the \(i\)-th polynomial shows that the completed
local ring of \(X^{r+1}\) at \(\eta\) is isomorphic to
\(\kk(x_1,\ldots,x_m)[\!\![z]\!\!]/(\epsilon z^{q+1})\) for a unit
\(\epsilon\), concluding the proof.
\end{proof}

Double-counting now determines the number of isotropic Hermitian \(k\)-planes.
The result is best expressed in terms of Gaussian integers. The parameter
here will be taken to be \(\bar{q} \coloneqq -q\), following the Ennola duality
principle \cite{Ennola}, which says that representation-theoretic
information about the finite unitary group \(\mathrm{GU}_{n+1}(q)\) can be
obtained from that of the finite general linear group \(\mathrm{GL}_{n+1}(q)\)
via the change of parameters \(q \mapsto \bar{q}\). To set notation, given
positive integers \(n\) and \(k\), write
\[
[n]_{\bar q} \coloneqq \frac{1 - \bar{q}^n}{1 - \bar{q}},
\qquad
[n]_{\bar q}! \coloneqq \prod\nolimits_{i = 1}^n [i]_{\bar q},
\qquad
\stirlingI{n}{k}_{\bar q} \coloneqq
\frac{[n]_{\bar q}!}{[k]_{\bar q}! [n-k]_{\bar q}!},
\]
and
\([2k+1]_{\bar q}!! \coloneqq \prod\nolimits_{i = 0}^k [2i+1]_{\bar q}\). Then
for each \(0 \leq k < n/2\), the result is as follows:

\begin{Proposition}\label{hermitian-planes-count}
\(\displaystyle
\#\mathbf{F}_k(X)_{\mathrm{Herm}} =
(1 - \bar{q})^{k+1} [2k+1]_{\bar q}!! \stirlingI{n+1}{2k+2}_{\bar q}
\).
\end{Proposition}

\begin{proof}
Write \(n\) as either \(2m+1\) or \(2m+2\), and consider the set
\[
\Lambda \coloneqq
\Set{(\PP U \subseteq \PP W) |
[\PP U] \in \mathbf{F}_k(X)_{\mathrm{Herm}}
\;\text{and}\;
[\PP W] \in \mathbf{F}_m(X)_{\mathrm{Herm}}}
\]
consisting of nested isotropic Hermitian \(k\)- and \(m\)-planes. Fixing
either \(\PP U\) or \(\PP W\) yields bijections
\[
\Lambda \simeq
\mathbf{F}_k(X)_{\mathrm{Herm}} \times \mathbf{F}_{m-k-1}(\bar{X})_{\mathrm{Herm}}
\quad\text{and}\quad
\Lambda
\simeq
\mathbf{F}_m(X)_{\mathrm{Herm}} \times \mathbf{G}(k+1,m+1)(\mathbf{F}_{q^2}),
\]
respectively; here, \(\bar{X}\) is a smooth \(q\)-bic hypersurface of dimension
\(n-2k-3\). These work as follow:

For the first, by \parref{incidences-containing-a-plane} and
\parref{basics-incidences-nondegenerate} together with the compatibility
of taking Hermitian vectors with taking orthogonal complements from
\citeForms{2.2}, the set of Hermitian \(m\)-planes in
\(X\) containing a fixed Hermitian \(k\)-plane \(\PP U\) is in bijection with
the set of Hermitian \((m-k-1)\)-planes in the \(q\)-bic hypersurface in
\(\PP(U^\dagger/U)\) defined by the nondegenerate \(q\)-bic form induced by
\(\beta\); choosing an isomorphism with the fixed smooth \(q\)-bic
\((n-2k-3)\)-fold \(\bar{X}\) gives the first bijection.

For the second, the discussion of \parref{hermitian-coordinates} implies that
the set of Hermitian \(k\)-planes contained in the fixed Hermitian \(m\)-plane
\(\PP W\) is in bijection with the set of \(\mathbf{F}_{q^2}\)-rational
\(k\)-planes in a projective \(m\)-space. Thus a choice of identification with
a fixed \(\PP^m\) gives the second bijection.

Taking cardinality of \(\Lambda\) and rearranging now gives an expression
for \(\#\mathbf{F}_k(X)_{\mathrm{Herm}}\) in terms of maximal isotropic Hermitian
subspaces, counted by \parref{hermitian-maximal-count}, and
\[
\#\mathbf{G}(k+1,m+1)(\mathbf{F}_{q^2})
= \stirlingI{m+1}{k+1}_{q^2}
= \frac{\prod_{i = 1}^{m+1}(1-\bar{q}^{2i})}
  {\prod_{i = 1}^{k+1} (1-\bar{q}^{2i}) \prod_{i = 1}^{m-k} (1-\bar{q}^{2i})},
\]
see \cite[Proposition 1.7.2]{Stanley:ECI}. A straightforward
computation now gives the result.
\end{proof}

\subsectiondash{Cyclic planes}\label{cyclic-planes}
Let \(0 \leq k < n/2\). Consider the closed subscheme of the Fano scheme
\[
\FF_{k,\mathrm{cyc}} \coloneqq
\Set{[P] \in \FF_k | \dim P \cap \phi_X(P) \geq k - 1}
\]
parameterizing \(k\)-planes in \(X\) whose intersection with its
\(\phi\)-translate is at least a \((k-1)\)-plane. The geometric construction
of the dynamical filtration from \parref{hermitian-filtration} provides a
natural rational map
\[
\phi^{[0,k]} \colon X^k \dashrightarrow \FF_{k,\mathrm{cyc}}
\qquad
x \mapsto \langle x, \phi_X(x), \ldots, \phi_X^k(x) \rangle
\]
which sends a general point of \(X^k\) to the cyclic \(k\)-plane it generates
via \(\phi\); this is well-defined since, by \parref{hermitian-minimal}, its
indeterminacy locus is the union of the Hermitian \((k-1)\)-planes of \(X\).

Resolve \(\phi^{[0,k]}\) as follows: For each \(0 \leq r \leq k\), consider
the following closed subschemes of partial flag varieties associated with \(V\):
\begin{align*}
X^k_r & \coloneqq
\Set{(L_0 \subset \cdots \subset L_r) |
\phi(L_i) \subset L_{i+1},
\;\text{and}\;
\langle L_r, \phi(L_r), \ldots, \phi^{k-r}(L_r)\rangle\;\text{is isotropic}}, \\
\FF_k^{k-r} & \coloneqq
\Set{(L_r \subset \cdots \subset L_k) |
L_i \subset \phi(L_{i+1}),\;
L_k\;\text{is isotropic},
\;\text{and}\;
\dim L_k \cap \phi(L_k) \geq k},
\end{align*}
where \(L_i\) denotes an \((i+1)\)-dimensional linear subspace of \(V\).
Then \(X^k = X_0^k\) and \(\FF_{k,\mathrm{cyc}} = \FF_k^0\). Set
\(\tilde{X}^k \coloneqq X^k_k\) and
\(\tilde{\FF}_{k,\mathrm{cyc}} \coloneqq \FF_k^k\).
Forgetting the last or first piece of the flag provides a series of morphisms
\[
\tilde{X}^k \to X^k_{k-1} \to \cdots \to X^k_1 \to X^k
\quad\text{and}\quad
\tilde{\FF}_{k,\mathrm{cyc}} \to \FF_k^{k-1} \to \cdots \to \FF_k^1 \to \FF_{k,\mathrm{cyc}}.
\]
Denote by \(\pi_{\smallbullet} \colon \tilde{X}^k \to X^k\) and
\(\pi^{\smallbullet} \colon \tilde{\FF}_{k,\mathrm{cyc}} \to \FF_{k,\mathrm{cyc}}\)
the projections from the top of the tower to the bottom. Finally,
the moduli description of the final pieces provides morphisms
\(\phi^- \colon \tilde{X}^k \to \tilde{\FF}_{k,\mathrm{cyc}}\)
and
\(\phi^+ \colon \tilde{\FF}_{k,\mathrm{cyc}} \to \tilde{X}^k\)
that send a flag \((L_i)_{i = 0}^k\) to \((\phi^{k-i}(L_i))_{i = 0}^k\) and
\((\phi^i(L_i))_{i = 0}^k\), respectively.

The result is summarized in the following. The construction is perhaps a
unitary analogue of the wonderful compactification of Drinfeld's upper half
space, as explained in \cite[\S\S7--9]{Langer:Drinfeld}. Compare also with
Ekedahl's analysis in \cite[Proposition 2.4]{Ekedahl} of Hirokado's variety
from \cite{Hirokado}.

\begin{Theorem}\label{cyclic-resolve}
Let \(0 \leq k < \frac{n}{2}\). The morphisms
\[
\pi_{\smallbullet} \colon \tilde{X}^k \to X^k
\quad\text{and}\quad
\pi^{\smallbullet} \colon \tilde{\FF}_{k,\mathrm{cyc}} \to \FF_{k,\mathrm{cyc}}
\]
are proper and birational from smooth irreducible schemes of dimension
\(n-k-1\); the morphisms
\[
\phi^- \colon \tilde{X}^k \to \tilde{\FF}_{k,\mathrm{cyc}}
\quad\text{and}\quad
\phi^+ \colon \tilde{\FF}_{k,\mathrm{cyc}} \to \tilde{X}^k
\]
are finite, flat, and purely inseparable of degrees
\(q^{k(k+1)}\) and \(q^{k(2n-3k-3)}\), respectively; and the diagram
\[
\begin{tikzcd}[column sep=3em]
\tilde{X}^k
  \rar["\phi^-"']
  \dar["\pi_{\smallbullet}^{\phantom{\smallbullet}}"'] &
\tilde{\FF}_{k,\mathrm{cyc}}
  \rar["\phi^+"']
  \dar["\pi_{\phantom{\smallbullet}}^{\smallbullet}"'] &
\tilde{X}^k
  \dar["\pi_{\smallbullet}"] \\
X^k
  \rar[dashed,"\phi^{[0,k]}"] &
\FF_{k,\mathrm{cyc}}
  \rar[dashed,"\phi^{(0 | k)}"] &
X^k
\end{tikzcd}
\]
is commutative,
where \(\phi^{(0|k)} \colon \FF_{k,\mathrm{cyc}} \dashrightarrow X^k\)
is the dominant rational map given by \(P \mapsto P \cap \phi^k(P)\).
\end{Theorem}

This is proved in parts through the remainder of this Section.

\begin{Lemma}\label{cyclic-blowups}
Let \(0 \leq r \leq k\). The map \(\pi_r \colon X^k_{r+1} \to X^k_r\) is the
blowup along the strict transform of the Hermitian \(r\)-planes of \(X\), and
its exceptional divisor is
\[
E_r \coloneqq
\Set{(L_0 \subset \cdots \subset L_{r+1}) \in X^k_{r+1}
| L_r \;\text{is Hermitian}} =
\mathrm{V}
\big(
\phi^*(\mathcal{L}_r/\mathcal{L}_{r-1}) \to
\mathcal{L}_{r+1}/\mathcal{L}_r\big).
\]
The map
\(\pi^{k-r} \colon \FF_{k,\mathrm{cyc}}^{k-r+1} \to \FF_{k,\mathrm{cyc}}^{k-r}\)
is the blowup along the strict transform of the locus
of \(k\)-planes containing a Hermitian \(r\)-plane
in \(\FF_{k,\mathrm{cyc}}\), and its exceptional divisor is
\[
E^{k-r} \coloneqq
\Set{(L_{r-1} \subset \cdots \subset L_k) \in \FF_{k,\mathrm{cyc}}^{k-r+1} |
L_r\;\text{is Hermitian}} =
\mathrm{V}\big(
\mathcal{L}_r/\mathcal{L}_{r-1} \to
\phi^*(\mathcal{L}_{r+1}/\mathcal{L}_r)
\big).
\]
\end{Lemma}

\begin{proof}
Let \(x\) be a point of \(X^k_{r+1}\), corresponding to a flag
\((L_i)_{i = 0}^{r+1}\). If \(x\) lies away from \(E_r\), then \(L_{r+1}\)
is spanned by \(L_r\) and \(\phi(L_r)\), whence \(\pi_r\) is an isomorphism
away from \(E_r\). An equation for \(E_r\) is obtained by noting that
\(L_r\) is Hermitian if and only if the canonical projection \(\phi(L_r) \to L_{r+1}/L_r\)
vanishes; since \(\phi(L_{r-1}) \subset L_r\), this projection descends
to the quotient and vanishes if and only if the induced map
\[
\phi(L_r)/\phi(L_{r-1}) \to L_{r+1}/L_r
\]
vanishes. Passing to the universal subbundles exhibits \(E_r\) as the
claimed effective Cartier divisor.

Let \(x\) now be a point of \(\FF^{k-r+1}_{k,\mathrm{cyc}}\), corresponding
to a flag \((L_i)_{i = r-1}^k\). This time, \(L_{r-1}\) is the intersection of
\(L_r\) and \(\phi(L_r)\) as long as \(x\) lies away from \(E^{k-r}\), whence
\(\pi^{k-r}\) is an isomorphism away from \(E^{k-r}\). Since
\(L_r \subset \phi(L_{r+1})\) and \(L_{r-1} \subset \phi(L_r)\), arguing
analogously shows that \(L_r\) is Hermitian if and only if the canonically
induced map
\[ L_r/L_{r-1} \to \phi(L_{r+1})/\phi(L_r) \]
vanishes, thereby establishing that \(E^{k-r}\) is the stated effective Cartier
divisor.
\end{proof}

The next statement identifies the tangent space to either \(\tilde{X}_k\) or
\(\tilde{\FF}_{k,\mathrm{cyc}}\) at a point corresponding to a flag
\(L_{\smallbullet} \coloneqq (L_i)_{i = 0}^k\) with residue field \(\kappa\):

\begin{Lemma}\label{cyclic-smoothness}
The tangent space of \(\tilde{X}^k\) at \(L_{\smallbullet}\) is isomorphic to
the vector space
\[
\Big(\bigoplus\nolimits_{i = 0}^{k-1} \Hom_\kappa(L_i/\phi(L_{i-1}), L_{i+1}/L_i)\Big)
\oplus \Hom_\kappa(L_k/\phi(L_{k-1}), L_k^{[1],\perp}/L_k).
\]
The tangent space of \(\tilde{\FF}_{k,\mathrm{cyc}}\) at a flag
\(L_{\smallbullet}\) is isomorphic to the vector space
\[
\Big(\bigoplus\nolimits_{i = 0}^{k-1}
\Hom_\kappa(L_i/L_{i-1}, \phi(L_{i+1})/L_i)\Big)
\oplus
\Hom_\kappa(L_k/L_{k-1}, L_k^{[1],\perp}/L_k).
\]
\end{Lemma}

\begin{proof}
Identify the tangent space to \(\tilde{X}^k\) at \(L_{\smallbullet}\) as a
subspace of the tangent space to the ambient partial flag variety, given by
\[
\Big(\bigoplus\nolimits_{i = 0}^{k - 1} \Hom_\kappa(L_i, L_{i+1}/L_i)\Big) \oplus
\Hom_\kappa(L_k,V/L_k).
\]
Let \(\varphi_{\smallbullet} \coloneqq (\varphi_i)_{i = 0}^k\) be a tangent
vector of the flag variety, and let
\(\tilde{L}_{\smallbullet} \coloneqq (\tilde{L}_i)_{i = 0}^k\) be the
corresponding first-order deformation. Then \(\varphi_{\smallbullet}\) is
tangent to \(\tilde{X}^k\) if and only if
\[
\tilde{L}_i \supset \phi(\tilde{L}_{i-1}) =
\phi(L_{i-1}) \otimes_\kappa \kappa[\epsilon]
\;\;\text{for}\; 1 \leq i \leq k,
\]
and \(L_k\) remains isotropic. The first set of conditions means that
the linear maps \(\varphi_i \colon L_i \to L_{i+1}/L_i\) vanish on
\(\phi(L_{i-1}) \subset L_i\); and the second condition means that,
as in \parref{differential-identify}, \(\varphi_k\) factors through
\(L_k^{[1],\perp}/L_k\). This identifies the tangent space to \(\tilde{X}^k\)
as the first subspace in the statement.

Identify the tangent space to \(\tilde{\FF}_{k,\mathrm{cyc}}\) at
\(L_{\smallbullet}\) directly: A first-order deformation
\(\tilde{L}_{\smallbullet}\) satisfies
\[
\tilde{L}_i \subset
\phi(\tilde{L}_{i+1}) =
\phi(L_{i+1}) \otimes_\kappa \kappa[\epsilon]
\;\;\text{for}\; 0 \leq i \leq k-1,
\]
and \(\tilde{L}_k\) is isotropic for \(\beta\). For each \(0 \leq i \leq k\),
encode \(\tilde{L}_i\) in the usual way as the \(\kappa[\epsilon]\)-module
spanned in \(\phi(L_{i+1}) \otimes_\kappa \kappa[\epsilon]\)
by an arbitrary lift of the image of a linear map
\(L_i \to \phi(L_{i+1})/L_i\); here, write \(L_{k+1} = V\).
The flag condition \(\tilde{L}_{i-1} \subset \tilde{L}_i\)
means that this linear map is determined on the subspace \(L_{i-1} \subset L_i\),
and so \(\tilde{L}_i\) is completely determined from \(\tilde{L}_{i-1}\) by a
linear map \(\varphi_i \colon L_i/L_{i-1} \to \phi(L_{i+1})/L_i\). When
\(i = k\), the condition that \(\tilde{L}_k\) is isotropic
means that \(\varphi_k\) further factors through the subspace
\(L_k^{[1],\perp}/L_k\). This identifies the set of first-order deformations of
\(L_{\smallbullet}\) in \(\tilde{\FF}_{k,\mathrm{cyc}}\) with the second
vector space in the statement.
\end{proof}

\begin{Corollary}\label{cyclic-smooth-irreducible}
The schemes \(\tilde{X}^k\) and \(\tilde{\mathbf{F}}_{k,\mathrm{cyc}}\) are
smooth and irreducible of dimension \(n-k-1\).
\end{Corollary}

\begin{proof}
The morphisms \(\phi^\pm\) defined in
\parref{cyclic-planes} compose to the finite purely inseparable endomorphism
\(\phi^2\) on either \(\tilde{X}^k\) and \(\tilde{\FF}_{k,\mathrm{cyc}}\).
Whence the \(\phi^\pm\) share the same properties, and in particular,
\(\tilde{X}^k\) and \(\tilde{\FF}_{k,\mathrm{cyc}}\) are universally
homeomorphic. Thus it suffices to establish the topological properties
for \(\tilde{X}^k\).

The dimension statement follows from \parref{cyclic-blowups} and
\parref{hermitian-complete-intersection}, which together imply that
\(\tilde{X}^k\) is a blowup of the codimension \(k+1\) complete intersection
\(X^k\). Combined with \parref{cyclic-smoothness}, this implies that \(\tilde{X}^k\)
and \(\tilde{\FF}_{k,\mathrm{cyc}}\) are smooth. Finally, for irreducibility,
it suffices to show that \(\tilde{X}^k\) is connected, and this follows
from the fact that it is proper birational to the normal connected scheme
\(X^k\).
\end{proof}

\subsectiondash{}\label{cyclic-proof}
To complete the proof of \parref{cyclic-resolve}, it remains to establish
the statements about \(\phi^\pm\).

First, it follows from the proof of \parref{cyclic-smooth-irreducible} that
the rational map \(\phi^{[0,k]} \colon X^k \dashrightarrow \FF_{k,\mathrm{cyc}}\)
is dominant: in other words, the general point of \(\FF_{k,\mathrm{cyc}}\)
parameterizes a cyclic \(k\)-plane. Therefore there exists a rational map
\(\phi^{(0|k)} \colon \FF_{k,\mathrm{cyc}} \dashrightarrow X^k\) that takes
a cyclic \(k\)-plane \(P\) to the intersection \(P \cap \phi^k(P)\) with
its \(k\)-th \(\phi\)-translate. It now follows from the moduli descriptions
that the diagram of \parref{cyclic-resolve} commutes.

Second, that the \(\phi^\pm\) are finite purely inseparable was already
explaind in the proof of \parref{cyclic-smooth-irreducible}. That they are
flat follows from miracle flatness, \citeSP{00R4}, since \(\tilde{X}^k\) and
\(\tilde{\mathbf{F}}_{k,\mathrm{cyc}}\) are smooth by
\parref{cyclic-smoothness}.

Finally, it remains to determine the degrees of \(\phi^\pm\). By flatness, it
suffices to compute the degree of the fibre over any given point. Consider
a point of \(\tilde{\FF}_{k,\mathrm{cyc}}\) corresponding to a flag
\((L_i)_{i = 0}^k\) of isotropic Hermitian subspaces and let
\(\kappa\) be its residue field. Then its preimage along
\(\phi^- \colon \tilde{X}^k \to \tilde{\FF}_{k,\mathrm{cyc}}\)
is the scheme representing the deformation functor
whose value on an Artinian local \(\kappa\)-algebra \(A\) is the set of flags
in \(V_\kappa \otimes_\kappa A\) of the form
\[
\Set{
(\tilde{L}_0 \subset \tilde{L}_1 \subset \cdots \subset \tilde{L}_k) |
\tilde{L}_i \otimes_A \kappa = L_i
\;\text{and}\;
\phi^{k-i}(\tilde{L}_i) = L_i \otimes_\kappa A\;
\text{for each}\; 0 \leq i \leq k
}.
\]
Since \(A\) is local, each \(\tilde{L}_i\) is free; moreover, the second condition
for \(i = k\) means that \(\tilde{L}_k = L_k \otimes_\kappa A\). So if
\(L_k = \langle v_0,\ldots,v_k\rangle\) is a basis of Hermitian vectors adapted
to the filtration, then a deformation is of the form:
\[
\tilde{L}_{k-r} =
\Big\langle v_i + \sum\nolimits_{j = 1}^r a_{ij} v_{k+1-j}: 0 \leq i \leq r \Big\rangle
\quad\text{where}\;
\phi^j(a_{ij}) = a_{ij}^{q^{2j}} = 0.
\]
This identifies the fibre over \((L_i)_{i = 0}^k\) along \(\phi^-\) with the
affine scheme
\[
\Spec\big(
\kappa[a_{ij} : 0 \leq i \leq k-1, 1 \leq j \leq k-i]/(a_{ij}^{q^{2j}})\big),
\]
and so \(\deg(\phi^-) = q^{k(k+1)}\). Finally,
\[
\deg(\phi^+) =
\deg(\phi^+ \circ \phi^-) \cdot
\deg(\phi^-)^{-1} =
\deg(\phi^k_{\tilde{X}^k}) \cdot \deg(\phi^-)^{-1} =
q^{k(2n-3k-3)}.
\]
This completes the proof of \parref{cyclic-resolve}.
\qed

\medskip

Consider the case \(n = 2m+2\) and \(k = m\), that is, the case of maximal
isotropic subspaces in an odd-dimensional \(q\)-bic hypersurface. By
\parref{cyclic-blowups} and \parref{cyclic-smooth-irreducible}, the scheme
\(\FF_{m,\mathrm{cyc}}\) is irreducible of dimension \(m+1\); comparing with
Theorem \parref{theorem-fano-schemes} now implies that
\(\FF_{m,\mathrm{cyc}} = \FF_m\). Therefore \parref{cyclic-resolve} implies:

\begin{Corollary}\label{cyclic-maximal}
Let \(n = 2m+2\). Then \(\FF_m\) is dominated via a purely inseparable
rational map of degree \(q^{m(m+1)}\) by a complete intersection geometrically
isomorphic to
\[
\pushQED{\qed}
\Set{(x_0:x_1: \cdots : x_{2m+2}) \in \PP^{2m+2} |
x_0^{q^{2i+1}+1} + x_1^{q^{2i+1}+1} + \cdots + x_{2m+2}^{q^{2i+1}+1} = 0\;\;
\text{for}\; 0 \leq i \leq m}.
\qedhere
\popQED
\]
\end{Corollary}

\section{Deligne--Lusztig stratification}\label{section-dl}
As described in \parref{hermitian-structures}, a nonsingular \(q\)-bic form
\(\beta\) on \(V\) induces an isogeny
\(F \colon \mathbf{GL}_V \to \mathbf{GL}_V\) which squares to the Frobenius
endomorphism associated with the \(\mathbf{F}_{q^2}\)-rational structure
\(\sigma_\beta\) on \(V\). This isogeny acts on the variety \(\mathbf{Bor}\)
of Borel subgroups in \(\mathbf{GL}_V\); intersecting its graph with
the orbits of \(\mathbf{GL}_V\) in \(\mathbf{Bor} \times_\kk \mathbf{Bor}\)
gives rise to the Deligne--Lusztig varieties from \cite{DL} in type
\({}^2\mathrm{A}_n\), those associated with the finite unitary group
\(\mathrm{U}(V,\beta)\). The aim of this Section is to relate the Fano
schemes \(\FF\) of \(r\)-planes in the smooth \(q\)-bic hypersurface
\(X\) associated with \(\beta\) with (generalized) Deligne--Lusztig varieties,
and to use this relation to access the \'etale cohomology of the Fano schemes:
see \parref{dl-betti}.

\subsectiondash{Deligne--Lusztig varieties}\label{dl-classical}
Let \(\mathbf{Bor}\) be the smooth projective variety parameterizing Borel
subgroups of \(\mathbf{GL}_V\). Diagonally acting via conjugation by
\(\mathbf{GL}_V\) decomposes the self-product
\[
\mathbf{Bor} \times_\kk \mathbf{Bor} =
\bigsqcup\nolimits_{w \in \mathrm{W}} \mathbf{orb}_w
\]
into orbits indexed by the Weyl group \(\mathrm{W}\). Let \(F \colon
\mathbf{Bor} \to \mathbf{Bor}\) be the morphism of \(\kk\)-varieties induced by
the isogeny \(F\). The \emph{Deligne--Lusztig variety} indexed by \(w\) is the
scheme-theoretic intersection
\[
\mathbf{DL}_w \coloneqq \mathbf{orb}_w \cap \Gamma_F
\]
of the \(w\)-orbit and the graph of the morphism \(F\). Each \(\mathbf{DL}_w\)
can be identified with a locally closed subscheme of \(\mathbf{Bor}\) via the
first projection, is smooth, and is of dimension the length of \(w\) in
\(\mathrm{W}\). The closure relations amongst the \(\mathbf{DL}_w\) are
given by the Bruhat order: see \cite[\S1]{DL} for this and more.

More generally, let \(\mathrm{S}\) be the set of simple reflections in
\(\mathrm{W}\), and for each \(I \subseteq \mathrm{S}\), let \(\mathrm{W}_I\) be
the subgroup generated by \(I\) and fix a set \({}^I\mathrm{W}\) of minimal
length representatives of the cosets \(\mathrm{W}_I\backslash \mathrm{W}\). Let
\(\mathbf{Par}_I\) be the set of parabolic subgroups of type \(I\).
Lusztig defines in \cite[\S4]{Lusztig:PartialFlag} a stratification
\[
\mathbf{Par}_I = \bigsqcup\nolimits_{w \in {}^I\mathrm{W}} \mathbf{DL}_{I,w}
\]
generalizing the classical Deligne--Lusztig stratification when
\(\mathbf{Par}_\varnothing = \mathbf{Bor}\); see also \cite{Bedard,
He:PartialFlag}. By \cite[Theorem 3.1]{He:PartialFlag}, the
\emph{generalized Deligne--Lusztig variety} \(\mathbf{DL}_{I,w}\) can be
obtained as the image of \(\mathbf{DL}_w\) under the projection
\(\mathbf{Bor} \to \mathbf{Par}_I\).

\subsectiondash{}\label{dl-flag-variety}
Describe the \(\mathbf{DL}_w\) geometrically by identifying \(\mathbf{Bor}\)
with the complete flag variety \(\mathbf{Flag}_V\): Fix an \(F\)-stable Borel
subgroup \(\mathbf{B}\) of \(\mathbf{GL}_V\), and let
\(V_{\smallbullet} \coloneqq (V_i)_{i = 0}^{n+1}\) be the complete flag it
stabilizes. Since \(F^2 = \phi\) on \(\mathbf{GL}_V\), each \(V_i\) is fixed by
\(\phi\) and is thus Hermitian in the sense of \parref{hermitian-subspaces}.
A computation analogous to the one in \parref{dl-F}
below shows that
\[
V_{n+1-i} = V_i^{[1],\perp}
\;\;
\text{for each}\;
0 \leq i \leq n+1,
\]
so that the first half of the flag is isotropic.
Identifying the \(g\)-conjugate of \(\mathbf{B}\) with the \(g\)-translate of
\(V_{\smallbullet}\) gives an isomorphism
\(\mathbf{Bor} \cong \mathbf{Flag}_V\). Abusing notation, write \(F\) for the
induced endomorphism of \(\mathbf{Flag}_V\). Then \(\mathbf{DL}_w\) can be
described as the locally closed subscheme of the flag variety given by
\[
\mathbf{DL}_w =
\Set{[W_{\smallbullet}] \in \mathbf{Flag}_V |
\gr^{F(W_{\smallbullet})}_{w(i)} \gr^{W_{\smallbullet}}_i(V) \neq \{0\}}
\]
where the Weyl group \(\mathrm{W}\) is viewed as the symmetric group on
\(\{1,\ldots,n+1\}\), and
\(\gr_i^{W_{\smallbullet}}\) is the \(i\)-th graded piece associated with
the filtration \(W_{\smallbullet}\): see, for example, \cite[\S2.1]{DL}.
A simple but useful instance of what this means is:
\(w(\set{1,\ldots,i}) \subseteq \set{1,\ldots,j}\) if and only if
\(W_i \subseteq F(W_{\smallbullet})_j\).

The action of the endomorphism \(F\) on a flag \(W_{\smallbullet}\) is
described as follows:

\begin{Lemma}\label{dl-F}
\(F\big(\{0\} \subsetneq W_1 \subsetneq \cdots \subsetneq W_n \subsetneq V\big) =
\big(\{0\} \subsetneq W_n^{[1],\perp} \subsetneq \cdots \subsetneq W_1^{[1],\perp} \subsetneq V\big)\).
\end{Lemma}

\begin{proof}
Write \(W_{\smallbullet} = g \cdot V_{\smallbullet}\) for \(g \in \mathbf{GL}_V\).
Then \(F(W_i) = F(g) \cdot V_i = (\beta^{-1} \circ g^{[1],\vee,-1} \circ \beta)(V_i)\).
First,
\[
\beta(V_i) =
(V^{[1]}/V_i^\perp)^\vee
= (V/V_i^{[1],\perp})^{[1],\vee}
= (V/V_{n+1-i})^{[1],\vee}
\]
with the middle equality due to
\parref{hermitian-subspaces}\ref{hermitian-subspaces.perps}, since \(V_i\)
is Hermitian. Therefore
\[
F(W_i)
= (\beta^{-1} \circ g^{[1],\vee,-1})\big((V/V_{n+1-i})^{[1],\vee}\big)
= \beta^{-1}\big((V/W_{n+1-i})^{[1],\vee}\big)
= W_{n+1-i}^{[1],\perp}.
\qedhere
\]
\end{proof}

More generally, given a subset \(I \subseteq \mathrm{S}\) of simple
reflections, let \(\mathbf{P}_I\) be the unique parabolic subgroup of type
\(I\) containing \(\mathbf{B}\), and use this to identify
\(\mathbf{Par}_I\) with the partial flag variety \(\mathbf{Flag}_{I,V}\) of
type \(I\). Then the generalized Deligne--Lusztig varieties give a partition
\[
\mathbf{Flag}_{I,V}
= \bigsqcup\nolimits_{w \in {}^I\mathrm{W}} \mathbf{DL}_{I,w}.
\]
The schemes \(\mathbf{DL}_{I,w}\) can sometimes be described in terms of
intersection patterns between a partial flag and its orthogonal by writing it as
the image of a Deligne--Lusztig variety from the complete flag variety. Compare
the following with \cite[\S2.3]{HLZ}:


\begin{Lemma}\label{dl-fano-schemes-are-dl}
The Fano scheme \(\mathbf{F}\) is a union of generalized Deligne--Lusztig varieties of type
\({}^2\mathrm{A}_{n+1}\).
\end{Lemma}

\begin{proof}
Since \(W_{r+1}^{[1],\perp} = F(W_{n-r})\) by \parref{dl-F}, the description of
\(\mathbf{DL}_w\) from \parref{dl-flag-variety} implies that
\[
W_{r+1} \subseteq W_{r+1}^{[1],\perp}
\;
\iff
w(\set{1,\ldots,r+1}) \subseteq \set{1,\ldots,n-r}.
\]
In other words, whether or not \(W_{r+1}\) is isotropic is determined by the
Weyl group element \(w\). Since the \(\mathbf{DL}_{I,w}\) are images of the
\(\mathbf{DL}_w\) under the projection
\(\mathbf{Flag}_V \to \mathbf{G}\), it follows that \(\mathbf{F}\)
is the union of generalized Deligne--Lusztig varieties associated with \(\mathbf{G}\).
\end{proof}

Two cases for which the Deligne--Lusztig stratification of \(\FF\) is particularly
simple and useful are: first, the hypersurface \(X\) itself; second, and more
interestingly, the Fano scheme of maximal isotropic subspaces. In the
following, denote by \(\mathrm{S} = \set{s_1,\ldots,s_n}\) the usual set of
simple transpositions which generates the Weyl group, identified with the
symmetric group on \(n+1\) elements as above.

\subsectiondash{Hypersurface}\label{dl-hypersurface}
Consider the stratification associated with the hypersurface \(X\).
Projective space is the partial flag variety associated with the set
\(I \coloneqq \set{s_2,\ldots,s_n}\) of simple reflections. A set of minimal
length representatives for \(\mathrm{W}_I \backslash \mathrm{W}\) can be taken
to be
\[
{}^I\mathrm{W} \coloneqq
\set{
\id,\;
s_1,\;
s_1s_2,\;
\ldots,\;
s_1s_2 \cdots s_n}.
\]
View the strata \(\mathbf{DL}_{I, w}\) as images of the \(\mathbf{DL}_w\)
under the projection \(\mathbf{Flag}_V \to \PP V\). Then a computation using
the descriptions from \parref{dl-flag-variety} and \parref{dl-F} relates
the generalized Deligne--Lusztig varieties in this case with the pieces
of the dynamical filtration from \parref{hermitian-filtration}: for
\(0 \leq k \leq n\),
\[
\mathbf{DL}_{I,s_1 \cdots s_k}
= X^{n-k-1} \setminus X^{n-k}
\;\;
\]
where \(X^{-1} \coloneqq \PP V\), and \(X^{n-k}\) is the union of the Hermitian
\(k\)-planes in \(X\) for \(0 \leq k \leq n/2\). This decomposition can
be combined with the methods of \cite{DL, Lusztig:Unipotent} to determine the
\'etale cohomology of \(X\) as a representation for the projective unitary
group. The following statement was first observed by Tate and Thompson in
\cite[p.102]{Tate:Conjecture}, and later made more precise by \cite{HM:TT-Lemma}:

\begin{Theorem}\label{dl-hypersurface-cohomology}
Assume \(\kk\) is separably closed. The middle primitive \'etale cohomology
\(\mathrm{H}^{n-1}_{\mathrm{\acute{e}t}}(X,\mathbf{Q}_\ell)_{\mathrm{prim}}\) is
an irreducible representation for \(\mathrm{U}_{n+1}(q)\) of dimension
\((-1)^n \bar{q} [n]_{\bar{q}}\).
\qed
\end{Theorem}

\subsectiondash{Maximal isotropic subspaces}\label{dl-maximal-isotropics}
Let \(n = 2m+1\) and let \(\FF\) be the Fano scheme of \(m\)-planes in \(X\).
Then Theorem \parref{theorem-fano-schemes} shows that \(\FF\) is a finite set
of reduced points; its dimension can also be obtained from
\parref{dl-fano-schemes-are-dl} by showing there is exactly one generalized
Deligne--Lusztig variety contained in \(\FF\).

The remainder of this Section is concerned with the case \(n = 2m+2\) and the
Deligne--Lusztig decomposition of the Fano scheme \(\FF\) of \(m\)-planes in
\(X\). The ambient Grassmannian \(\mathbf{G}(m+1,2m+3)\) is the partial
flag variety associated with the set
\[
I
= \set{
s_1,
\ldots,\;
s_m,\;
s_{m+2},
\ldots,\;
s_{2m+2}}
= \mathrm{S} \setminus \{s_{m+1}\}.
\]
As in \parref{dl-fano-schemes-are-dl}, the stratum \(\mathbf{DL}_{I,w}\) is
contained in \(\FF\) if and only if
\(w(\set{1,\ldots,m+1}) \subseteq \set{1,\ldots,m+2}\). Since \(\mathrm{W}_I\)
contains all permutations on the first \(m+1\) elements, a straightforward
computation shows that a set \({}^I\mathrm{W}\) of minimal length representatives
for the coset space \(\mathrm{W}_I\backslash \mathrm{W}\) satisfying this
condition consists of the following \(m+2\) cyclic permutations:
\(w_0 \coloneqq \id\) and, for \(1 \leq k \leq m+1\),
\[
w_k
\coloneqq s_{m+1} s_m \cdots s_{m+2-k}
= (m+2-k,\, m+3-k,\ldots,\,m+1,\,m+2).
\]

The next statement gives a moduli description of
\(\overline{\mathbf{DL}}_{I,w_k}\), and may be compared with \cite[Theorem
2.15]{Vollaard}. Note already that, since the \(w_k\) are linearly ordered in
the Bruhat order, the closure relations of the associated Deligne--Lusztig
varieties are also linear, giving the first identification in:

\begin{Proposition}\label{dl-description}
\(
\displaystyle
\overline{\mathbf{DL}}_{I,w_k} =
\bigsqcup\nolimits_{i = 0}^k \mathbf{DL}_{I,w_i} =
\Set{P \in \mathbf{F} |
\text{\(P\) contains a Hermitian \((m-k)\)-plane}}
\).
\end{Proposition}

\begin{proof}
Proceed by showing the more refined statement that the classical
Deligne--Lusztig variety \(\mathbf{DL}_{w_k}\) parameterizes complete flags
\((W_i)_{i = 0}^{n+1}\) satisfying:
\begin{enumerate}
\item\label{dl-description.isotropic}
\(W_i = W_{n+1-i}^{[1],\perp}\)
for \(m+2 \leq i \leq n+1\),
\item\label{dl-description.hermitian}
\(W_i = \phi(W_i)\)
for \(0 \leq i \leq m+1-k\), and
\item\label{dl-description.translate}
\(W_i \cap \phi(W_i) = W_{i-1}\)
for \(m+2-k \leq i \leq m+1\).
\end{enumerate}
In other words, \(\mathbf{DL}_{w_k}\) parameterizes complete isotropic
flags whose first \(m+1-k\) terms are Hermitian, and the remaining terms
can be determined by \(W_{m+1}\) and \(\phi\); in fact,
\ref{dl-description.translate} implies that
\[
W_i = \bigcap\nolimits_{j = 0}^{m+1-i} \phi^j(W_{m+1})
\quad\text{for}\; m+2-k \leq i \leq m+1.
\]
With \ref{dl-description.hermitian}, this gives
\(W_{m+1-k} = \bigcap\nolimits_{j \geq 0} \phi^j(W_{m+1})\) meaning this
is the maximal Hermitian subspace of \(W_{m+1}\) by \parref{hermitian-minimal}.
Projecting down to the Grassmannian would then show that \(\mathbf{DL}_{I,w_k}\)
parameterizes \(m\)-planes in \(X\) that contain a Hermitian \((m-k)\)-plane
but no \((m-k+1)\)-plane, implying the statement.

Describe \(\mathbf{DL}_{w_k}\) now by translating the equations imposed by
\(w_k\) to conditions on flags via \parref{dl-flag-variety}
and \parref{dl-F}: That the cycle \(w_k\) stabilizes the intervals
\(\set{1,\ldots,i}\) for \(m+2 \leq i \leq n+1\) is equivalent to condition
\ref{dl-description.isotropic}. Since \(F^2 = \phi\), this implies that
\[
F(W_{\smallbullet})_i
= (W_i^{[1],\perp})^{[1],\perp}
= \phi(W_i)
\;\;\text{for}\;0 \leq i \leq m+1.
\]
With this, that \(w_k\) acts trivially on \(\set{1,\ldots,i}\) with
\(0 \leq i \leq m+1-k\) is equivalent to \ref{dl-description.hermitian}.
Finally, that \(w_k(i) = i+1\) for \(m+2-k \leq i \leq m+1\) means that
\[
\gr^{F(W_{\smallbullet})}_{i+1}
\gr^{W_{\smallbullet}}_i(V) \neq \set{0}
\]
which, in light of what has been established so far, amounts to the two
conditions \(W_i \neq \phi(W_i)\), and
\(W_i \subset \phi(W_{i+1})\) when \(i \neq m+1\) and that \(W_{m+1}\) is
isotropic. Taken altogether, these conditions are equivalent to those in
\ref{dl-description.translate}. This now dispenses with all the equations
of \(\mathbf{DL}_{w_k}\).
\end{proof}

\subsectiondash{}\label{dl-irred-comps}
Combined with \parref{incidences-containing-a-plane} and Theorem
\parref{theorem-fano-schemes}, the description of \parref{dl-description}
implies that the irreducible
components of \(\overline{\mathbf{DL}}_{I,w_k}\) are indexed by the Hermitian
\((m-k)\)-planes in \(X\), and each component is isomorphic to the Fano scheme
of \(k\)-planes contained in a smooth \(q\)-bic hypersurface of dimension
\(2k+1\). Moreover, the irreducible component indexed by the isotropic
\(k\)-plane \(P\) is contained exactly in the irreducible components of
\(\overline{\mathbf{DL}}_{I,w_{k+1}}\) indexed by the isotropic
\((k+1)\)-planes containing \(P\).

\medskip

For each \(k \geq 0\), write \(\mathbf{DL}^{\mathrm{Cox}}_k\) for the open
dense Deligne--Lusztig stratum of the Fano scheme of \(k\)-planes in a smooth
\(q\)-bic hypersurface of dimension \(2k+1\); this is \(\mathbf{DL}_{w_{m+1}}\)
when \(k = m\) as above. The indexing Weyl group element contains exactly one
simple reflection in each \(F\)-orbit, and so it is a \emph{Coxeter element} in
type \({}^2\mathrm{A}_{2k+2}\), whence the notation; see
\cite[(1.7)]{Lusztig:Frobenius}. The discussion of \parref{dl-irred-comps}
gives a useful stratification of the Fano scheme, phrased in terms of the
Grothendieck ring \(\mathrm{K}_0(\mathrm{Var}_\kk)\) of varieties in the
following statement; compare it with \cite[Proposition 2.5.1]{HLZ}:

\begin{Corollary}\label{dl-decomposition}
\(\displaystyle
[\mathbf{F}] =
[\mathbf{DL}_{m+1}^{\mathrm{Cox}}] +
\sum\nolimits_{k = 0}^m
\#\mathbf{F}_{m-k}(X)_{\mathrm{Herm}} \cdot [\mathbf{DL}_k^{\mathrm{Cox}}]
\)
in \(\mathrm{K}_0(\mathrm{Var}_\kk)\).
\qed
\end{Corollary}

Notably, the zeta function of the Coxeter Deligne--Lusztig variety
\(\mathbf{DL}_k^{\mathrm{Cox}}\) with respect to their \(\FF_{q^2}\)-structure
\(\phi\) can be explicitly determined. In the following, notation is as in
\parref{hermitian-planes-count}:

\begin{Lemma}\label{dl-zeta-coxeter}
\(\displaystyle
\log \zeta(\mathbf{DL}_k^{\mathrm{Cox}}, t) =
\sum\nolimits_{i = 0}^{2k} (-1)^{i+1} \bar{q}^{\binom{2k+1-i}{2}}
\stirlingI{2k}{i}_{\bar{q}}
\log(1 - \bar{q}^it)
\).
\end{Lemma}

\begin{proof}
Lusztig's calculation in \cite[(6.1.2)]{Lusztig:Frobenius} gives
\[
t\frac{d}{dt}\log\zeta(\mathbf{DL}_k^{\mathrm{Cox}},t) =
\sum_{s \geq 1} \#\mathbf{DL}_k^{\mathrm{Cox}, \phi^s} \cdot t^s =
\bar{q}^{k(2k+1)} (1 - \bar{q})^{2k} [2k]_{\bar{q}}!
\cdot
\frac{t^{2k+1}}{\prod\nolimits_{i = 0}^{2k} (1 - \bar{q}^i t)}
\]
upon comparing with the tables \cite[p.106 and p.147]{Lusztig:Frobenius} and
noting that the associated adjoint group has order
\(\#\mathrm{PU}_{2k+1}(q) = \bar{q}^{k(2k+1)} \prod\nolimits_{i = 2}^{2k+1}(\bar{q}^i - 1)\).
The partial fraction decomposition of the rational function is given by:
\[
\frac{t^{2k+1}}{\prod\nolimits_{i = 0}^{2k} (1 - \bar{q}^it)} =
\sum_{i = 0}^{2k}\Bigg(
(-1)^i \bar{q}^{\binom{i+1}{2}-2ik}(1 - \bar{q})^{-2k}
\frac{1}{[2k]_{\bar{q}}!}
\stirlingI{2k}{i}_{\bar{q}}
\frac{t}{1 - \bar{q}^i t}
\Bigg).
\]
Combining and applying the operator
\(\int \frac{dt}{t}\) then gives the result.
\end{proof}

These together give an explicit expression for the \'etale Betti numbers of the
Fano scheme:

\begin{Theorem}\label{dl-betti}
For each \(0 \leq k \leq 2m+2\),
\[
b_k(\mathbf{F}) =
\bar{q}^{\binom{2m+3-k}{2}} \stirlingI{2m+2}{k}_{\bar{q}} +
\sum_{i = 0}^{m - \lceil  k/2\rceil} (1-\bar{q})^{i+1}
\bar{q}^{\binom{2m+1-2i-k}{2}}
[2i+1]_{\bar{q}}!!
\stirlingI{2m+3}{2i+2}_{\bar{q}}
\stirlingI{2m-2i}{k}_{\bar{q}}.
\]
\end{Theorem}

\begin{proof}
The logarithmic zeta function is additive along disjoint unions, so this
follows from the motivic decomposition of \(\mathbf{F}\) in
\parref{dl-decomposition} together with the computations
\parref{hermitian-planes-count} and \parref{dl-zeta-coxeter}.
\end{proof}

Since \(\FF\) is smooth of dimension \(m+1\) by Theorem \parref{theorem-fano-schemes},
Poincar\'e duality gives \(b_k(\FF) = b_{2m+2-k}(\FF)\) for each
\(0 \leq k \leq 2m+2\). This yields a curious set of identities amongst
\(q\)-binomial coefficients. In particular, this gives a neat expression for
the first Betti number:

\begin{Corollary}\label{dl-first-betti}
\(b_1(\FF) = b_{2m+1}(\FF) = \bar{q} [2m+2]_{\bar{q}}\). \qed
\end{Corollary}

Comparing with \parref{dl-hypersurface-cohomology} shows that \(b_1(\FF) = b_{2m+1}(X)\).
Adapting the argument of \citeThesis{2.8.3}
shows that the Fano correspondence \(\mathbf{L}\) between \(X\) and \(\FF\)
acting as in \parref{cgaj-action} induces an isomorphism
\[
\mathbf{L}^* \colon \mathrm{H}^{2m+1}_{\mathrm{\acute{e}t}}(X,\mathbf{Z}_\ell)
\to \mathrm{H}^1_{\mathrm{\acute{e}t}}(\FF,\mathbf{Z}_\ell).
\]
Details are omitted, as this will not be used in the following.

\section{Intermediate Jacobian}\label{section-cgaj}
This Section is concerned with finer aspects of the geometry of a smooth
\(q\)-bic hypersurface of odd dimension \(2m+1\); as such, the base field
\(\kk\) will be henceforth taken to be algebraically closed. The main result is
that the Albanese variety \(\mathbf{Alb}_\FF\) of the Fano scheme \(\FF\) of
\(m\)-planes in \(X\) is purely inseparably isogeneous via an Abel--Jacobi map
to a certain intermediate Jacobian \(\mathbf{Ab}_X^{m+1}\) associated with
\(X\): see \parref{cgaj-result}. Here, \(\mathbf{Ab}_X^{m+1}\) is taken to be
the algebraic representative for algebraically-trivial cycles of codimension
\(m+1\) in \(X\), see \parref{cgaj-algrep}. Existence of
\(\mathbf{Ab}_X^{m+1}\) is known in this case by work of Murre and Fakhruddin,
see \parref{cgaj-algrep-exists}. The result relies on a study of the geometry
of special cycles in \(\FF\), which occupies the first half of this Section,
the main computations being \parref{cgaj-point-divisor-intersections} and
\parref{cgaj-P-and-x}.

\subsectiondash{Incidence divisors}\label{cgaj-incidence-divisors}
The Fano scheme \(\FF\) contains two particularly interesting types of
divisors that arise from incidence relations; first are the divisors \(D_P\) of
\(m\)-planes in \(X\) that are incident with a fixed \(m\)-plane
\(P \subset X\) as in \parref{correspondence-incidence}; and second are the
divisors
\[
D_x \coloneqq \Set{[P] \in \FF | P\;\text{contains}\; x}
\]
of \(m\)-planes that contain a fixed Hermitian point \(x\) of the hypersurface
\(X\). It follows from \parref{incidences-containing-a-plane} and
\parref{basics-incidences-nondegenerate} that \(D_x\) is isomorphic to the Fano
scheme of \((m-1)\)-planes in a smooth \(q\)-bic hypersurface of dimension
\(2m-1\); moreover, writing \(L \subset V\) for the \(1\)-dimensional subspace
underlying \(x\), \(D_x\) may be obtained from \(\FF\) via intersection with
the Grassmannian \(\mathbf{G}(m,V/L)\), linearly embedded in
\(\mathbf{G}(m+1,V)\) as the subvariety of \((m+1)\)-dimensional subspaces
containing \(L\).

The next few paragraphs study the relations amongst these divisors. To begin,
consider how the divisors associated with two Hermitian points \(x\) and \(y\)
intersect. To state the result, phrased in terms of the associated line
bundles, one piece of notation: Suppose that \(x\) and \(y\) span a Hermitian
line \(\ell\) in \(X\). Then \(D_x\) is the Fano scheme of \((m-1)\)-planes in
the \(q\)-bic hypersurface \(\bar{X}\) lying at the base of the cone
\(X \cap \mathbf{T}_{X,x}\) obtained by intersecting \(X\) with its embedded
tangent space \(\mathbf{T}_{X,x}\) at \(x\). The line \(\ell\) corresponds to a
Hermitian point of \(\bar{X}\), denoted by \(\bar{y}\). The divisors \(D_x\)
and \(D_y\) intersect as follows:

\begin{Lemma}\label{cgaj-point-divisor-intersections}
Let \(x\) and \(y\) be Hermitian points of \(X\). Then
\[
\sO_{\FF}(D_y)\rvert_{D_x} \cong
\begin{dcases*}
\sO_{D_x}(1-q) & if \(x = y\), \\
\sO_{D_x}(D_{\bar{y}}) & if \(x \neq y\) and \(\langle x, y \rangle \subset X\), and \\
\sO_{D_x} & otherwise.
\end{dcases*}
\]
\end{Lemma}

\begin{proof}
The case \(x = y\) pertains to the normal bundle of \(D_x\) in \(\FF\), which
may be determined via \parref{differential-generic-smoothness} as
\[
\sO_{\FF}(D_x)\rvert_{D_x}
\cong \omega_{D_x} \otimes \omega_\FF^\vee\rvert_{D_x}
\cong \sO_{D_x}(1-q)
\]
since \(D_x\) is linearly embedded in \(\FF\) with respect to the Pl\"ucker
polarization. Now suppose that \(x \neq y\). Let \(L\) and \(M\) be the
\(1\)-dimensional subspaces of \(V\) underlying \(x\) and \(y\), respectively.
By the discussion of \parref{cgaj-incidence-divisors}, \(D_x \cap D_y\) is
the intersection of the Fano scheme with two sub-Grassmannian varieties in
\(\mathbf{G}(m+1,V)\), as on the the left side of:
\[
\FF \cap \mathbf{G}(m,V/L) \cap \mathbf{G}(m,V/M)
= \FF \cap \mathbf{G}(m-1,V/\langle L,M \rangle).
\]
The sub-Grassmannians intersect along the Grassmannian of \((m-1)\)-dimensional
subspaces containing the \(2\)-dimensional space \(\langle L,M \rangle\).
This implies that \(D_x\) and \(D_y\) intersect along \(D_{\bar{y}}\) if the
line \(\langle x, y \rangle\) is contained in \(X\), and are disjoint otherwise.
\end{proof}

The next task is to relate the divisors \(D_P\) and \(D_x\). The main
observation is that Hermitian planes intersect \(m\)-planes in Hermitan
subplanes:

\begin{Proposition}\label{cgaj-hermitian-incidence}
Let \(P_0 \subset X\) be a Hermitian \(k\)-plane for some \(0 \leq k \leq m\).
If \(P' \subset X\) is any \(m\)-plane, then \(P_0 \cap P'\) is either empty or
a Hermitian plane.
\end{Proposition}

\begin{proof}
Begin with three observations: First, it follows from
\parref{hermitian-complete-intersection} or, alternatively, the arguments of
\parref{hermitian-planes-count} that every Hermitian \(k\)-plane in \(X\) is
contained in a Hermitian \(m\)-plane in \(X\), so it suffices to consider the
case \(k = m\). Second, if \(P'\) is Hermitian, then \(P_0 \cap P'\)
is stable under \(\phi\) and is Hermitian by \parref{hermitian-subspaces}.
Third, suppose \(P_0 \cap P'\) contains a nonempty Hermitian plane; let
\(K \subseteq V\) be a nonzero Hermitian subspace corresponding to a maximal
such. Let \(\bar{X}\) be the nonsingular \(q\)-bic hypersurface at the base of
the cone \(X \cap \PP K^\dagger\) as in
\parref{incidences-containing-a-plane} and
\parref{basics-incidences-nondegenerate}, and let \(\bar{P}_0\) and
\(\bar{P}'\) be the corresponding planes in \(\bar{X}\). Since taking Hermitian
vectors respects orthogonal decompositions, as in the proof of
\parref{hermitian-planes-count}, it follows that \(\bar{P}_0\) is a Hermitian
plane in \(\bar{X}\) and \(\bar{P}_0 \cap \bar{P}'\) does not contain a
Hermitian plane. Thus it suffices to consider the case where \(P_0\) is a
Hermitian \(m\)-plane, \(P'\) is not Hermitian, and where \(P_0 \cap P'\) does
not contain any Hermitian planes, in which case the task is to show that
\(P_0 \cap P' = \varnothing\).

Assume, on the contrary, that \(P_0 \cap P'\) is an \(r\)-plane with \(r \geq 0\),
so that \(P_0\) and \(P'\) together span an \((2m-r)\)-plane \(L\) in \(\PP V\).
Observe that for each \(i \geq 0\), the \(m\)-plane \(\phi^i(P')\) is isotropic
by \parref{hermitian-phi-properties}\ref{hermitian-phi-properties.isotropic}
and is moreover contained in \(L\): Indeed, as follows from
\parref{cyclic-planes} with the comments preceding \parref{cyclic-maximal},
\(\phi^i(P') \cap \phi^{i+1}(P')\) has dimension \(m-1\), and since \(P_0\) is
Hermitian while \(P_0 \cap P'\) is not,
\[
P_0 \cap \phi^{i+1}(P')
= \phi^{i+1}(P_0 \cap P')
\neq \phi^i(P_0 \cap P')
= P_0 \cap \phi^i(P').
\]
Here, the first and third equalities use the fact that \(\phi\) is injective
and stabilizes \(P_0\). Therefore \(\phi^{i+1}(P')\) is spanned by the
\((m-1)\)-plane \(\phi^i(P') \cap \phi^{i+1}(P')\), and any point in \(P_0 \cap
\phi^{i+1}(P')\)---which the assumption implies is not empty!--- not contained
in \(\phi^i(P')\). Induction now implies that \(\phi^{i+1}(P')\) lies in \(L\),
as claimed.

Consider the \(q\)-bic hypersurface \(X' \coloneqq X \cap L\) of dimension
\(2m-r-1\). Then \(X'\) contains the \(m\)-planes \(P_0\) and, by
the claim above, each of the \(\phi^i(P')\) with \(i \geq 0\). This implies
two things: first, since a smooth hypersurface can never contain linear spaces
larger than half its dimension, \(X'\) is singular; and, second, the Fano
scheme of \(m\)-planes in \(X'\) is everywhere of larger than its expected
dimension by \parref{basics-fano-equations}. Thus
\parref{differential-singular-locus} implies that each of \(P_0\) and the
\(\phi^i(P')\) must pass through the singular locus of \(X'\). But, finally, on
the one hand,
\[
\varnothing \neq
\Sing X' \subseteq
P_0 \cap \Big(\bigcap\nolimits_{i \geq 0} \phi^i(P')\Big).
\]
On the other hand, the intersection on the right is stable under \(\phi\),
whence is Hermitian \parref{hermitian-subspaces}. This contradicts the
assumption that \(P_0 \cap P'\) contains no Hermitian planes, so
\(P_0 \cap P' = \varnothing\).
\end{proof}

The next statement gives the basic relationship between the divisors
\(D_P\) and \(D_x\). Continuing with the notation set at the end of
\parref{cgaj-incidence-divisors}, given a Hermitian point \(x\) of \(X\),
write \(\bar{P}\) for the \((m-1)\)-plane induced by \(P\) in the \(q\)-bic
hypersurface \(\bar{X}\) at the base of the cone \(X \cap \mathbf{T}_{X,x}\).

\begin{Proposition}\label{cgaj-P-and-x}
Let \(P \subset X\) be an \(m\)-plane. Then for every Hermitian point
\(x\) of \(X\),
\[
\sO_{\FF}(D_P)\rvert_{D_x} \cong
\begin{dcases*}
\sO_{D_x}(D_{\bar{P}}) & if \(x \notin P\), and \\
\sO_{D_x}(D_{\bar{P}})^{\otimes q^2} \otimes_{\sO_{D_x}} \sO_{D_x}(1-q) & if \(x \in P\).
\end{dcases*}
\]
Moreover, there is a decomposition of effective Cartier divisors on \(\FF\)
given by
\[
D_P
= D_P' + \sum\nolimits_{x \in P \cap X_{\mathrm{Herm}}} D_x
\]
where the components of \(D_P'\) generically
parameterize \(m\)-planes \(X\) which are incident with \(P\) and are
disjoint from \(P \cap X_{\mathrm{Herm}}\).
\end{Proposition}

\begin{proof}
Fix a Hermitian point \(x\) of \(X\). If \(x \notin P\), then
\(D_P \cap D_x\) parameterizes \(m\)-planes in \(X\) which pass through the
vertex of the cone \(X \cap \mathbf{T}_{X,x}\) and which intersect the \((m-1)\)-plane
\(P \cap \mathbf{T}_{X,x}\). In other words, projecting from the vertex \(x\)
to obtain a \(q\)-bic hypersurface \(\bar{X}\), identifying
\(D_x\) as the Fano scheme of \((m-1)\)-planes in \(\bar{X}\), and
letting \(\bar{P}\) be the \((m-1)\)-plane induced in \(\bar{X}\), it follows
that
\[
\sO_\FF(D_P)\rvert_{D_x} \cong \sO_{D_x}(D_{\bar{P}}).
\]

If \(x \in P\), then it follows from \parref{incidences-containing-a-plane}
that \(D_P - D_x\) is an effective Cartier divisor. Smoothness of the Fano
correspondence \(\mathbf{L} \to X\) above \(x\) from
\parref{basics-incidences-nondegenerate} together with the construction of
\(D_P\) in \parref{correspondence-incidence} shows that \(D_P - D_x\) is the
closure of the image under \(\mathbf{L} \to \FF\) of
\[
\Set{
(z, [P']) \in \mathbf{L} |
P \cap P' \ni z \neq x
}.
\]
Therefore \(D_P - D_x\) intersects \(D_x\) at the \((m-1)\)-dimensional locus
parameterizing \(m\)-planes in \(X\) containing a line through \(x\) which
intersects \(P\). Projecting from \(x\), this is the divisor in \(D_x\)
parameterizing \((m-1)\)-planes in \(\bar{X}\) incident with \(\bar{P}\).
Therefore there is some \(a \in \mathbf{Z}\) such that
\[
\sO_\FF(D_P)\rvert_{D_x} \cong
\sO_\FF(D_P - D_x)\rvert_{D_x} \otimes_{\sO_{D_x}} \sO_\FF(D_x)\rvert{D_x} \cong
\sO_{D_x}(D_{\bar{P}})^{\otimes a} \otimes_{\sO_{D_x}} \sO_{D_x}(1-q),
\]
where the second identification is by
\parref{cgaj-point-divisor-intersections}.
Determine the multiplicity \(a\) by computing the degree of
\(\sO_\FF(D_P)\rvert_{D_x}\) with respect to the Pl\"ucker polarization of
\(D_x\) in two ways. On the one hand, this identification
together with the fact that \((q+1)D_{\bar{P}}\) is algebraically equivalent to
\(c_1(\bar{\mathcal{Q}}) = c_1(\sO_{D_x}(1))\) from \parref{correspondence-plucker} gives
\[
\deg(\sO_\FF(D_P)\rvert_{D_x}) =
a\deg(\sO_{D_x}(D_{\bar{P}})) + (1-q)\deg(\sO_{D_x}(1)) =
\frac{a + 1 - q^2}{q+1}\deg(\sO_{D_x}(1)).
\]
On the other hand, since \((q+1)D_P\) is algebraically equivalent to
\(c_1(\sO_\FF(1))\) again by \parref{correspondence-plucker}, and since the
Pl\"ucker polarizations of \(\FF\) and \(D_x\) are compatible, it follows that
\[
\deg(\sO_\FF(D_P)\rvert_{D_x}) = \frac{1}{q+1} \deg(\sO_{D_x}(1)).
\]
Comparing the two expressions shows that \(a = q^2\).
Finally, that \(D_P' \coloneqq D_P - \sum\nolimits_{x \in P \cap X_{\mathrm{Herm}}} D_x\)
is the effective Cartier divisor generically parameterizing \(m\)-planes
incident with \(P\) and disjoint from \(P \cap  X_{\mathrm{Herm}}\)
now follows from the discussion above together with \parref{cgaj-hermitian-incidence}.
\end{proof}

In the case \(P\) is Hermitian, \parref{cgaj-hermitian-incidence} implies that
the divisor \(D_P'\) appearing in \parref{cgaj-P-and-x} vanishes, and so:

\begin{Corollary}\label{cgaj-hermitian-m-plane}
Let \(P \subset X\) be a Hermitian \(m\)-plane. Then
\[
D_P = \sum\nolimits_{x \in P \cap X_{\mathrm{Herm}}} D_x
\]
as divisors on \(\FF\). The sum ranges over the Hermitian points of
\(X\) contained in \(P\).
\qed
\end{Corollary}

The Pl\"ucker degree of \(\FF\) can be easily determined combining
\parref{cgaj-hermitian-m-plane} and the proof method of \parref{cgaj-P-and-x};
this has also been recently computed in \cite[Theorem 1.1]{Li:DL} with
similar methods.

\begin{Corollary}\label{cgaj-plucker-degree}
\(\displaystyle
\deg_{\sO_\FF(1)}(\FF) = \prod\nolimits_{i = 0}^m \frac{q^{2i+2}-1}{q-1}\).
\end{Corollary}

\begin{proof}
Choose a Hermitian point \(x\) and a Hermitian \(m\)-plane \(P \subset X\)
containing it. Note that \(D_x\) is isomorphic to \(D_y\) for each
of the \(\#\PP^m(\FF_{q^2}) = \frac{q^{2m+2}-1}{q^2-1}\) points
\(y \in P \cap X_{\mathrm{Herm}}\), and that the Pl\"ucker polarization
of \(\FF\) is compatible with those on these divisors. Therefore, by
\parref{cgaj-hermitian-m-plane} and \parref{correspondence-plucker},
\begin{align*}
\frac{q^{2m+2}-1}{q^2-1} \deg(\sO_{D_x}(1))
& = \sum\nolimits_{y \in P \cap X_{\mathrm{Herm}}} \deg(\sO_\FF(1)\rvert_{D_x}) \\
& =
\int_\FF
  c_1(\sO_\FF(1))^m \cdot \Big(\sum\nolimits_{y \in P \cap X_{\mathrm{Herm}}} [D_y]\Big) \\
& = \int_\FF c_1(\sO_\FF(1))^m c_1(\sO_\FF(D_P))
= \frac{1}{q+1}\deg(\sO_\FF(1)).
\end{align*}
Since \(D_x\) is the Fano scheme of \((m-1)\)-planes on a smooth \(q\)-bic
hypersurface of dimension \(2m-1\), the result follows by induction upon
noting that, when \(m = 0\), \(\FF\) is a plane curve of degree
\(q+1 = \frac{q^2-1}{q-1}\).
\end{proof}

Finally, let \(R \subset X\) be a Hermitian \((m-1)\)-plane and consider
the restriction of \(\sO_{\FF}(D_P)\) to
\[
C_R \coloneqq \Set{[P] \in \FF | P\;\text{contains}\; R},
\]
the incidence scheme parameterizing \(m\)-planes in \(X\) containing \(R\).
This is a smooth \(q\)-bic curve by \parref{incidences-containing-a-plane} and
\parref{basics-incidences-nondegenerate}, and is even a complete intersection
of the divisors \(D_x\) as \(x\) ranges over a basis of Hermitian points in
\(R\) by \parref{cgaj-point-divisor-intersections}. The result is as follows:

\begin{Corollary}\label{cgaj-curves}
Let \(P \subset X\) be an \(m\)-plane and let \(R \subset X\) be a Hermitian \((m-1)\)-plane.
Then
\[
\sO_{\FF}(D_P)\rvert_{C_R} \cong
\sO_{C_R}([\bar{P}])^{\otimes q^{2k+2}} \otimes_{\sO_{C_R}}
\sO_{C_R}(-1)^{\otimes (q^{2k+2}-1)/(q+1)}
\]
where \(k \coloneqq \dim P \cap R\) and \([\bar{P}] \in C_R\) is the unique
\(m\)-plane in \(X\) containing \(R\) and incident with \(P \setminus R\).
\end{Corollary}

\begin{proof}
Choose a basis of Hermitian points
\(x_0,\ldots,x_k \in P \cap R \cap X_{\mathrm{Herm}}\)
and complete this to a basis of \(R\) with Hermitian points
\(x_{k+1},\ldots,x_{m-1} \in R \cap X_{\mathrm{Herm}}\). The result now
follows from successively applying \parref{cgaj-P-and-x} upon writing
\[
C_R = D_{x_0} \cap D_{x_1} \cap \cdots \cap D_{x_{m-1}}
\]
as the complete intersection of the divisors associated with these Hermitian
points.
\end{proof}

An important special case is when \(P \cap R\) is properly contained in
\(P \cap P_0\) for some Hermitian \(m\)-plane \(P_0\) containing \(R\):

\begin{Corollary}\label{cgaj-curves-P0}
Let \(P \subset X\) be an \(m\)-plane and let \(R \subset X\) be a Hermitian
\((m-1)\)-plane. Assume there exists a Hermitian \(m\)-plane \(P_0 \subset X\)
containing \(R\) such that \(P \cap R \subsetneq P \cap P_0\). Then
\[
\sO_{\FF}(D_P)\rvert_{C_R} \cong \sO_{C_R}([P_0]).
\]
\end{Corollary}

\begin{proof}
Since \(P_0\) contains \(R\) and is incident with \(P\) away from \(P \cap R\),
\(\bar{P} = P_0\) by uniqueness in \parref{cgaj-curves}. Since taking Hermitian
vectors is compatible with orthogonal decompositions, as in the proof of
\parref{hermitian-planes-count}, \([P_0]\) is a Hermitian point of \(C_R\).
Then by \parref{incidences-hermitian-perp} and the comments that precede it,
\(\sO_{C_R}(1) \cong \sO_{C_R}([P_0])^{\otimes q+1}\). Combining this with
\parref{cgaj-curves} now gives the result.
\end{proof}

The remainder of this Section will use these results regarding divisors on
\(\FF\) to relate the the Albanese variety of \(\FF\) with a certain abelian
variety \(\mathbf{Ab}_X^{m+1}\) attached to the \(q\)-bic hypersurface \(X\),
referred to as the \emph{intermediate Jacobian of \(X\)}: see
\parref{cgaj-result}. The definitions here follow \cite[\S3.2]{Beauville:Prym}
and \cite[\S\S1.5--1.8]{Murre:Jacobian}, but see also \cite{ACMV:Functorial,
ACMV:Descending} for a more modern perspective that works in more generality.

\subsectiondash{Algebraic representatives}\label{cgaj-algrep}
Let \(Y\) be a smooth projective variety over an algebraically closed field
\(\kk\). For each integer \(0 \leq k \leq \dim Y\), denote by
\(\mathrm{CH}^k(Y)_{\mathrm{alg}}\) the subgroup of algebraically trivial
codimension \(k\) cycle classes on \(Y\). Given an abelian variety \(A\) over
\(\kk\), a group homomorphism
\[
\phi \colon
\mathrm{CH}^k(Y)_{\mathrm{alg}} \to
A(\kk)
\]
is said to be \emph{regular} if for every pointed smooth projective variety
\((T,t_0)\) over \(\kk\), and every cycle class
\(Z \in \mathrm{CH}^k(T \times Y)\), the map
\[
T(\kk) \to \mathrm{CH}^k(Y)_{\mathrm{alg}} \to A(\kk),
\qquad
t \mapsto \phi(Z_t - Z_{t_0})
\]
is induced by a morphism \(T \to A\) of varieties over \(\kk\). Suppose
there exists a regular homomorphism
\[
\phi^k_Y \colon
\mathrm{CH}^k(Y)_{\mathrm{alg}} \to
\mathbf{Ab}_Y^k(\kk)
\]
that is initial amongst all regular homomorphisms from
\(\mathrm{CH}^k(Y)_{\mathrm{alg}}\). Then the pair
\((\mathbf{Ab}^k_Y, \phi^k_Y)\) is called an \emph{algebraic representative}
for codimension \(k\) cycles in \(Y\).

The basic general existence results regarding algebraic representatives are as
follows:

\begin{Theorem}\label{cgaj-algrep-exists}
Let \(Y\) be a smooth projective variety of dimension \(d\) over \(\kk\). Then
an algebraic representative for codimension \(k\) cycles exists when
\begin{enumerate}
\item\label{cgaj-algrep-exists.alb}
\(k = d\), and \(\mathbf{Ab}_Y^d = \mathbf{Alb}_Y\),
\item\label{cgaj-algrep-exists.pic}
\(k = 1\), and \(\mathbf{Ab}_Y^1 = \mathbf{Pic}_{Y,\mathrm{red}}^0\),
\item\label{cgaj-algrep-exists.jac}
\(k = 2\), and
\(2\dim \mathbf{Ab}_Y^2 \leq \dim_{\mathbf{Q}_\ell} \mathrm{H}^3_{\mathrm{\acute{e}t}}(Y,\mathbf{Q}_\ell)\), and
\item\label{cgaj-algrep-exists.supersingular}
the rational Chow motive of \(Y\) is a summand of the motive of a supersingular
abelian variety, in which case \(2\dim\mathbf{Ab}_Y^k = \dim_{\mathbf{Q}_\ell}\mathrm{H}^{2k-1}_{\mathrm{\acute{e}t}}(Y,\mathbf{Q}_\ell)\).
\end{enumerate}
\end{Theorem}

\begin{proof}
For \ref{cgaj-algrep-exists.alb}
and \ref{cgaj-algrep-exists.pic}, see
\cite[\S1.4]{Murre:Jacobian}.
For \ref{cgaj-algrep-exists.jac},
see \cite{Murre:Jacobian-CR} or \cite[Theorem 1.9]{Murre:Jacobian}, along with
a correction by \cite{Kahn:Jacobian}. For \ref{cgaj-algrep-exists.supersingular},
see \cite[Theorem 2]{Fakhruddin:Supersingular}, though notice the misprint
with the dimension statement, which may be corrected by comparing with
\cite[\S7]{Murre:Jacobian}.
\end{proof}

Returning to the situation of a smooth \(q\)-bic hypersurface \(X\) of
dimension \(2m+1\), an algebraic representative \(\mathbf{Ab}_X^{m+1}\) in
codimension \(m+1\) is referred to as the \emph{intermediate Jacobian} of
\(X\). Since \(X\) is isomorphic to a supersingular Fermat variety, this exists
by \parref{cgaj-algrep-exists}\ref{cgaj-algrep-exists.supersingular}; when \(m
= 1\), so that \(X\) is a smooth \(q\)-bic threefold, this also follows from
the general existence result
\parref{cgaj-algrep-exists}\ref{cgaj-algrep-exists.jac}.

The next statement relates cycles on \(X\) with those on \(\FF\) via the
Fano correspondence \(\mathbf{L}\) from \parref{cgaj-action}:

\begin{Lemma}\label{cgaj-intermediate-jacobian-morphisms}
There exists a commutative diagram of abelian groups
\[
\begin{tikzcd}
  \CH^{m+1}(\FF)_{\mathrm{alg}} \rar["\mathbf{L}_*"'] \dar["\phi^{m+1}_{\FF}"]
& \CH^{m+1}(X)_{\mathrm{alg}} \rar["\mathbf{L}^*"'] \dar["\phi^{m+1}_X"]
& \CH^1(\FF)_{\mathrm{alg}} \dar["\phi^1_{\FF}"] \\
  \mathbf{Alb}_{\FF}(\kk) \rar
& \mathbf{Ab}^{m+1}_X(\kk) \rar
& \mathbf{Pic}^0_{\FF}(\kk)
\end{tikzcd}
\]
and hence morphisms of abelian varieties
\[
\mathbf{Alb}_{\FF} \xrightarrow{\mathbf{L}_*}
\mathbf{Ab}^{m+1}_X \xrightarrow{\mathbf{L}^*}
\mathbf{Pic}^0_{\FF, \mathrm{red}}.
\]
\end{Lemma}

\begin{proof}
The action of the Fano correspondence gives the top row of maps, see
\parref{cgaj-action}. The vertical maps are the universal
regular homomorphisms recognizing the Albanese variety of \(\FF\), the
intermediate Jacobian of \(X\), and the Picard variety of \(\FF\) as
the algebraic representatives for algebraically trivial, respectively,
\(0\)-cycles of \(\FF\), \(m\)-cycles of \(X\), and \(m\)-cycles of
\(\FF\): see \parref{cgaj-algrep-exists}. The morphisms of the group
schemes arise from the corresponding universal property of each scheme.
\end{proof}

Fix a Hermitian \(m\)-plane \(P_0 \subset X\) and consider the Albanese
morphism \(\operatorname{alb}_{\FF} \colon \FF \to \mathbf{Alb}_\FF\)
centred at \([P_0]\). Composing this with the morphisms of abelian schemes
from \parref{cgaj-intermediate-jacobian-morphisms} yields a morphism
\(\FF \to \mathbf{Pic}_{\FF}^0\). Its action on \(\kk\)-points is
easily understood:

\begin{Lemma}\label{cgaj-intjac-F-to-Pic}
The morphism \(\FF \to \mathbf{Pic}_{\FF}^0\) acts on \(\kk\)-points by
\[
[P] \mapsto \sO_{\FF}(D_P - D_{P_0}).
\]
\end{Lemma}

\begin{proof}
The Albanese morphism on \(\kk\)-points factorizes as
\[
\phi^{m+1}_\FF \circ \operatorname{alb}_{\FF}(\kk) \colon
\FF(\kk) \to \CH^{m+1}(\FF)_{\mathrm{alg}} \to \mathbf{Alb}_{\FF}(\kk)
\]
where the first map is \([P] \mapsto [P] - [P_0]\), and the second map is
the universal regular homomorphism from \parref{cgaj-algrep-exists}\ref{cgaj-algrep-exists.alb}.
The commutative diagram from \parref{cgaj-intermediate-jacobian-morphisms}
then shows that \(\FF \to \mathbf{Pic}^0_{\FF}\) factors through the map
\(\FF(\kk) \to \mathrm{CH}^1(\FF)_{\mathrm{alg}}\) given by
\[
[P] \mapsto
\mathbf{L}^*\mathbf{L}_*([P] - [P_0]) =
[D_P] - [D_{P_0}],
\]
where the actions of \(\mathbf{L}\) are as in \parref{cgaj-action}. Composing
with
\(\phi^1_{\FF} \colon \CH^1(\FF)_{\mathrm{alg}} \to \mathbf{Pic}_{\FF}^0(\kk)\)
gives the result.
\end{proof}

With the fixed Hermitian \(m\)-plane \(P_0 \subset X\), let \(C_0 \subset \FF\)
be the closed subscheme parameterizing \(m\)-planes containing a Hermitian
\((m-1)\)-plane in \(P_0\); in other words, \(C_0\) is the union of the smooth
\(q\)-bic curves \(C_R\) from \parref{cgaj-curves}, where \(R\) ranges over
all Hermitian \((m-1)\)-planes contained in \(P_0\). Let
\[
C \coloneqq \coprod\nolimits_{R \subset P_0} C_R
\]
be the disjoint union of the irreducible components of \(C\) and let
\(\nu \colon C \to C_0\) be the normalization morphism. The following relates
the abelian varieties appearing in
\parref{cgaj-intermediate-jacobian-morphisms} with the Jacobian
\[
\mathbf{Jac}_C \coloneqq \prod\nolimits_R \mathbf{Jac}_{C_R}
\]
of the curve \(C\), defined as the product of the Jacobians of its connected
components:

\begin{Proposition}\label{cgaj-multiplication}
The composite morphism
\[
\Phi \colon
\mathbf{Jac}_C \xrightarrow{\nu_*}
\mathbf{Alb}_{\FF} \xrightarrow{\mathbf{L}_*}
\mathbf{Ab}_X^{m+1} \xrightarrow{\mathbf{L}^*}
\mathbf{Pic}^0_{\FF,\mathrm{red}} \xrightarrow{\nu^*}
\mathbf{Jac}_C
\]
is multiplication by \(q^{2m}\).
\end{Proposition}

\begin{proof}
Consider a \(\kk\)-point of a connected component \(C_R\) of \(C\) and identify
it with a \(\kk\)-point \([P]\) of its image in \(\FF\). Let
\(\mathrm{alb}_C \colon C \to \mathbf{Jac}_C\) be the Albanese map of \(C\),
which on the connected component \(C_R\), is the usual Abel--Jacobi map
centred at \([P_0]\) into the \(R\)-th factor of \(\mathbf{Jac}_C\) and
the constant map onto the identity otherwise. Combined with
\parref{cgaj-intjac-F-to-Pic}, this means that the map
\(\Phi \circ \mathrm{alb}_C \colon C \to \mathbf{Jac}_C\) acts as
\[
\Phi(\mathrm{alb}_C([P]))
= \nu^*\sO_\FF(D_P - D_{P_0}).
\]
Since the image of \(C\) under \(\mathrm{alb}_C\) generates \(\mathbf{Jac}_C\),
that \(\Phi\) is multiplication by \(q^{2m}\) will follow upon showing that,
for each Hermitian \((m-1)\)-plane \(R' \subset P_0\),
\[
\sO_\FF(D_P - D_{P_0})\rvert_{C_{R'}} \cong
\begin{dcases*}
\sO_{C_{R'}} & if \(R' \neq R\), and \\
\sO_{C_R}([P] - [P_0])^{\otimes q^{2m}} & if \(R' = R\).
\end{dcases*}
\]
Note that \parref{cgaj-curves-P0} implies
\(\sO_\FF(D_{P_0})\rvert_{C_{R'}} \cong \sO_{C_{R'}}([P_0])\) for all \(R'\).
When \(R' \neq R\), since \(P \cap R' \subsetneq P \cap P_0\), applying
\parref{cgaj-curves-P0} once more gives the conclusion. When \(R' = R\),
\parref{cgaj-curves} gives
\[
\sO_{\FF}(D_P)\rvert_{C_R} \cong
\sO_{C_R}([P])^{\otimes q^{2m}} \otimes_{\sO_{C_R}} \sO_{C_R}(-1)^{\otimes (q^{2m}-1)/(q+1)}.
\]
The conclusion now follows upon using the fact
\(\sO_{C_R}(1) \cong \sO_{C_R}([P_0])^{\otimes q+1}\) from
\parref{incidences-hermitian-perp}.
\end{proof}

Putting everything together shows that each of the abelian varieties in
question are related to one another via purely inseparable isogenies:

\begin{Theorem}\label{cgaj-result}
Each of the morphisms of abelian varieties
\[
\nu_* \colon \mathbf{Jac}_C \to \mathbf{Alb}_{\FF}, \quad
\mathbf{L}_* \colon \mathbf{Alb}_{\FF} \to \mathbf{Ab}_X^{m+1},
\quad
\mathbf{L}^* \colon \mathbf{Ab}_X^{m+1} \to \mathbf{Pic}_{\FF,\mathrm{red}}^0,
\quad
\nu^* \colon \mathbf{Pic}^0_{\FF,\mathrm{red}} \to \mathbf{Jac}_C
\]
is a purely inseparable \(p\)-power isogeny.
\end{Theorem}

\begin{proof}
Each of the abelian varieties \(\mathbf{Ab}_X^{m+1}\), \(\mathbf{Alb}_{\FF}\),
\(\mathbf{Pic}_{\FF,\mathrm{red}}^0\), and \(\mathbf{Jac}_C\) are
of dimension
\[
\frac{1}{2}\bar{q}[2m+2]_{\bar{q}} =
\frac{1}{2}\dim_{\mathbf{Q}_\ell} \mathrm{H}^{2m+1}_{\mathrm{\acute{e}t}}(X,\mathbf{Q}_\ell) =
\frac{1}{2}\dim_{\mathbf{Q}_\ell} \mathrm{H}^1_{\mathrm{\acute{e}t}}(\FF,\mathbf{Q}_\ell).
\]
This is
\parref{cgaj-algrep-exists}\ref{cgaj-algrep-exists.supersingular}
for \(\mathbf{Ab}_X^{m+1}\); this follows from \parref{dl-first-betti} for
the Picard and Albanese varieties of \(\FF\); and, for the Jacobian, compute:
\[
\dim\mathbf{Jac}_C =
\dim\mathbf{Jac}_{C_R} \cdot \#\check{\PP}^m(\FF_{q^2}) =
\frac{q(q-1)}{2} \cdot \frac{q^{2m+2}-1}{q^2-1} =
\frac{1}{2} q \cdot \frac{q^{2m+2}-1}{q+1} =
\frac{1}{2}\bar{q}[2m+2]_{\bar{q}}.
\]
Since \(\Phi\) is an \(p\)-power isogeny by \parref{cgaj-multiplication}, it
follows that each of the maps factoring \(\Phi\) are themselves \(p\)-power
isogenies. To see that each are purely inseparable, note that the smooth
\(q\)-bic curves are supersingular: this was known to Weil, as in
\cite{Weil:Fermat, Weil:Grossenchar}, though see also \cite[\S3]{SK:Fermat}.
Thus \(\mathbf{Jac}_C\), being the product of Jacobians of smooth \(q\)-bic
curves, is itself supersingular. Whence multiplication by \(q^{2m}\) is purely
inseparable, and so each of the constituent morphisms are as well.
\end{proof}

I feel that the case \(m = 1\) for \(q\)-bic threefolds is best understood in
analogy with cubic threefolds over the complex numbers: that the
maps in \parref{cgaj-multiplication} are purely inseparable isogenies means
that they are close to being isomorphisms; the statement about \(\mathbf{L}_*
\colon \mathbf{Alb}_\FF \to \mathbf{Ab}_X^2\) is analogous to the result of
Clemens and Griffiths in \cite[Theorem 11.19]{CG} saying that the Abel--Jacobi
map from the Albanese of the surface of lines on a complex cubic threefold is
an isomorphism onto the Hodge-theoretic intermediate Jacobian of the
hypersurface; the statement regarding
\(\nu_* \colon \mathbf{Jac}_C \to \mathbf{Alb}_S\) is an analogue of Mumford's
identification between the Albanese of the Fano surface with a
Prym variety in \cite[Appedix C]{CG}; and the statement regarding
\(\nu^* \circ \mathbf{L}^* \colon \mathbf{Ab}^2_X \to \mathbf{Jac}_C\) may be
seen as an analogue of Murre's identification in \cite[Theorem
10.8]{Murre:Prym} between the group of algebraically trivial \(1\)-cycles on a
cubic and the Prym. The latter two analogies will be explained further in
future work.

\bibliographystyle{amsalpha}
\bibliography{main}

\newcommand{\etalchar}[1]{$^{#1}$}
\providecommand{\bysame}{\leavevmode\hbox to3em{\hrulefill}\thinspace}
\providecommand{\MR}{\relax\ifhmode\unskip\space\fi MR }
\providecommand{\MRhref}[2]{%
  \href{http://www.ams.org/mathscinet-getitem?mr=#1}{#2}
}
\providecommand{\href}[2]{#2}
\begin{thebibliography}{ACMV23}

\bibitem[ACMV17]{ACMV:Descending}
Jeffrey~D. Achter, Sebastian Casalaina-Martin, and Charles Vial, \emph{On
  descending cohomology geometrically}, Compos. Math. \textbf{153} (2017),
  no.~7, 1446--1478.

\bibitem[ACMV23]{ACMV:Functorial}
\bysame, \emph{A functorial approach to regular homomorphisms}, Algebr. Geom.
  \textbf{10} (2023), no.~1, 87--129.

\bibitem[AK77]{AK:Fano}
Allen~B. Altman and Steven~L. Kleiman, \emph{Foundations of the theory of
  {F}ano schemes}, Compositio Math. \textbf{34} (1977), no.~1, 3--47.

\bibitem[BC66]{BC:Hermitian}
Raj~C. Bose and Indra~M. Chakravarti, \emph{Hermitian varieties in a finite
  projective space {${\rm PG}(N,\,q^{2})$}}, Canadian J. Math. \textbf{18}
  (1966), 1161--1182.

\bibitem[Bea77]{Beauville:Prym}
Arnaud Beauville, \emph{Vari\'{e}t\'{e}s de {P}rym et jacobiennes
  interm\'{e}diaires}, Ann. Sci. \'{E}cole Norm. Sup. (4) \textbf{10} (1977),
  no.~3, 309--391.

\bibitem[Bea90]{Beauville:Moduli}
\bysame, \emph{Sur les hypersurfaces dont les sections hyperplanes sont \`a
  module constant}, The {G}rothendieck {F}estschrift, {V}ol. {I}, Progr. Math.,
  vol.~86, Birkh\"{a}user Boston, Boston, MA, 1990, pp.~121--133.

\bibitem[B{\'e}d85]{Bedard}
Robert B{\'e}dard, \emph{On the {B}rauer liftings for modular representations},
  J. Algebra \textbf{93} (1985), no.~2, 332--353.

\bibitem[BPRS21]{BPRS:Lines}
Anna Brosowsky, Janet Page, Tim Ryan, and Karen~E. Smith, \emph{Geometry of
  {S}mooth {E}xtremal {S}urfaces}, preprint at
  \href{https://arxiv.org/abs/2110.15908}{arXiv:2110.15908} (2021).

\bibitem[BVdV79]{BVV:Fano}
Wolf Barth and Antonius Van~de Ven, \emph{Fano varieties of lines on
  hypersurfaces}, Arch. Math. (Basel) \textbf{31} (1978/79), no.~1, 96--104.

\bibitem[CG72]{CG}
C.~Herbert Clemens and Phillip~A. Griffiths, \emph{The intermediate {J}acobian
  of the cubic threefold}, Ann. of Math. (2) \textbf{95} (1972), 281--356.

\bibitem[Che22]{thesis}
Raymond Cheng, \emph{Geometry of q-bic {H}ypersurfaces}, ProQuest LLC, Ann
  Arbor, MI, 2022, Thesis (Ph.D.)--Columbia University.

\bibitem[Che23]{qbic-forms}
\bysame, \emph{$q$-bic forms}, preprint at
  \href{https://arxiv.org/abs/2301.09929}{arXiv:2301.09929} (2023).

\bibitem[Che24a]{qbic-threefolds}
\bysame, \emph{{$q$}-bic threefolds and their surface of lines}, preprint at
  \href{https://arxiv.org/abs/2402.09884}{arXiv:2402.09884} (2024).

\bibitem[Che24b]{ratconn}
\bysame, \emph{Minimal free rational curves in fano hypersurfaces have high
  degree}, in preparation (2024).

\bibitem[Col79]{Collino}
Alberto Collino, \emph{Lines on quartic threefolds}, J. London Math. Soc. (2)
  \textbf{19} (1979), no.~2, 257--267.

\bibitem[Con06]{Conduche}
Denis Conduch\'e, \emph{Courbes rationnelles et hypersurfaces de l'espace
  projectif}, Ph.D. thesis, Universit\'e Louis Pasteur, 2006.

\bibitem[Deb96]{Debarre:Connectivity}
Olivier Debarre, \emph{Th\'{e}or\`emes de connexit\'{e} pour les produits
  d'espaces projectifs et les grassmanniennes}, Amer. J. Math. \textbf{118}
  (1996), no.~6, 1347--1367.

\bibitem[Deb01]{Debarre:HDAG}
\bysame, \emph{Higher-dimensional algebraic geometry}, Universitext,
  Springer-Verlag, New York, 2001.

\bibitem[DK73]{SGAVII}
Pierre Deligne and Nicholas Katz, \emph{S\'{e}minaire de {G}\'{e}om\'{e}trie
  {A}lg\'{e}brique du {B}ois-{M}arie 1967--1969 - {G}roupes de {M}onodromie en
  {G}\'{e}om\'{e}trie {A}lg\'{e}brique. {II} {(SGA 7 II)}}, Lecture Notes in
  Mathematics, Vol. 340, Springer-Verlag, Berlin-New York, 1973.

\bibitem[DL76]{DL}
Pierre Deligne and George Lusztig, \emph{Representations of reductive groups
  over finite fields}, Ann. of Math. (2) \textbf{103} (1976), no.~1, 103--161.

\bibitem[DM98]{DM:Fano}
Olivier Debarre and Laurent Manivel, \emph{Sur la vari\'{e}t\'{e} des espaces
  lin\'{e}aires contenus dans une intersection compl\`ete}, Math. Ann.
  \textbf{312} (1998), no.~3, 549--574.

\bibitem[Eke03]{Ekedahl}
Torsten Ekedahl, \emph{On non-liftable {C}alabi--{Y}au threefolds},
  \href{https://arxiv.org/abs/math/0306435}{arXiv:0306435}, 2003.

\bibitem[Enn63]{Ennola}
Veikko Ennola, \emph{On the characters of the finite unitary groups.}, Ann.
  Acad. Sci. Fenn. Ser. A I No. (1963), 35.

\bibitem[Fak02]{Fakhruddin:Supersingular}
Najmuddin Fakhruddin, \emph{On the {C}how groups of supersingular varieties},
  Canad. Math. Bull. \textbf{45} (2002), no.~2, 204--212. \MR{1904084}

\bibitem[Ful98]{Fulton}
William Fulton, \emph{Intersection theory}, second ed., Ergebnisse der
  Mathematik und ihrer Grenzgebiete. 3. Folge. A Series of Modern Surveys in
  Mathematics, vol.~2, Springer-Verlag, Berlin, 1998.

\bibitem[He09]{He:PartialFlag}
Xuhua He, \emph{{$G$}-stable pieces and partial flag varieties}, Representation
  theory, Contemp. Math., vol. 478, Amer. Math. Soc., Providence, RI, 2009,
  pp.~61--70.

\bibitem[Hef85]{Hefez:Thesis}
Abramo Hefez, \emph{Duality for projective varieties}, Ph.D. thesis,
  Massachusetts Institute of Technology, 1985.

\bibitem[Hir79]{Hirschfeld:Geometries}
James W.~P. Hirschfeld, \emph{Projective geometries over finite fields}, Oxford
  Mathematical Monographs, The Clarendon Press, Oxford University Press, New
  York, 1979.

\bibitem[Hir99]{Hirokado}
Masayuki Hirokado, \emph{A non-liftable {C}alabi-{Y}au threefold in
  characteristic {$3$}}, Tohoku Math. J. (2) \textbf{51} (1999), no.~4,
  479--487.

\bibitem[HLZ19]{HLZ}
Xuhua He, Chao Li, and Yihang Zhu, \emph{Fine {D}eligne-{L}usztig varieties and
  arithmetic fundamental lemmas}, Forum Math. Sigma \textbf{7} (2019), e47, 55.

\bibitem[HM78]{HM:TT-Lemma}
Ryoshi Hotta and Kiyoshi Matsui, \emph{On a lemma of {T}ate-{T}hompson},
  Hiroshima Math. J. \textbf{8} (1978), no.~2, 255--268.

\bibitem[Huy23]{Huybrechts:Cubics}
Daniel Huybrechts, \emph{The geometry of cubic hypersurfaces}, Cambridge
  Studies in Advanced Mathematics, vol. 206, Cambridge University Press,
  Cambridge, 2023.

\bibitem[Kah21]{Kahn:Jacobian}
Bruno Kahn, \emph{On the universal regular homomorphism in codimension 2}, Ann.
  Inst. Fourier (Grenoble) \textbf{71} (2021), no.~2, 843--848.

\bibitem[KKP{\etalchar{+}}22]{KKPSSW:F-Pure}
Zhibek Kadyrsizova, Jennifer Kenkel, Janet Page, Jyoti Singh, Karen~E. Smith,
  Adela Vraciu, and Emily~E. Witt, \emph{Lower bounds on the {F}-pure threshold
  and extremal singularities}, Trans. Amer. Math. Soc. Ser. B \textbf{9}
  (2022), 977--1005.

\bibitem[Kol96]{Kollar:RationalCurves}
J\'{a}nos Koll\'{a}r, \emph{Rational curves on algebraic varieties}, Ergebnisse
  der Mathematik und ihrer Grenzgebiete. 3. Folge. A Series of Modern Surveys
  in Mathematics [Results in Mathematics and Related Areas. 3rd Series. A
  Series of Modern Surveys in Mathematics], vol.~32, Springer-Verlag, Berlin,
  1996.

\bibitem[KP91]{KP:Gauss}
Steven Kleiman and Ragni Piene, \emph{On the inseparability of the {G}auss
  map}, Enumerative algebraic geometry ({C}openhagen, 1989), Contemp. Math.,
  vol. 123, Amer. Math. Soc., Providence, RI, 1991, pp.~107--129.

\bibitem[Lan19]{Langer:Drinfeld}
Adrian Langer, \emph{Birational geometry of compactifications of {D}rinfeld
  half-spaces over a finite field}, Adv. Math. \textbf{345} (2019), 861--908.

\bibitem[Li23]{Li:DL}
Chao Li, \emph{Degrees of unitary {D}eligne--{L}usztig varieties}, preprint at
  \href{https://arxiv.org/abs/2301.08886}{arXiv:2301:08886} (2023).

\bibitem[LZ22]{LZ:Kudla}
Chao Li and Wei Zhang, \emph{Kudla-{R}apoport cycles and derivatives of local
  densities}, J. Amer. Math. Soc. \textbf{35} (2022), no.~3, 705--797.

\bibitem[LTX{\etalchar{+}}22]{LTXZZ}
Yifeng Liu, Yichao Tian, Liang Xiao, Wei Zhang, and Xinwen Zhu, \emph{On the
  {B}eilinson-{B}loch-{K}ato conjecture for {R}ankin-{S}elberg motives},
  Invent. Math. \textbf{228} (2022), no.~1, 107--375.

\bibitem[Lus76a]{Lusztig:Unipotent}
George Lusztig, \emph{On the finiteness of the number of unipotent classes},
  Invent. Math. \textbf{34} (1976), no.~3, 201--213.

\bibitem[Lus76b]{Lusztig:Green}
\bysame, \emph{On the {G}reen polynomials of classical groups}, Proc. London
  Math. Soc. (3) \textbf{33} (1976), no.~3, 443--475.

\bibitem[Lus07]{Lusztig:PartialFlag}
\bysame, \emph{A class of perverse sheaves on a partial flag manifold},
  Represent. Theory \textbf{11} (2007), 122--171.

\bibitem[Lus77]{Lusztig:Frobenius}
\bysame, \emph{Coxeter orbits and eigenspaces of {F}robenius}, Invent. Math.
  \textbf{38} (1976/77), no.~2, 101--159.

\bibitem[Mur72]{Murre:Prym}
Jacob~P. Murre, \emph{Algebraic equivalence modulo rational equivalence on a
  cubic threefold}, Compositio Math. \textbf{25} (1972), 161--206.

\bibitem[Mur83]{Murre:Jacobian-CR}
\bysame, \emph{Un r\'{e}sultat en th\'{e}orie des cycles alg\'{e}briques de
  codimension deux}, C. R. Acad. Sci. Paris S\'{e}r. I Math. \textbf{296}
  (1983), no.~23, 981--984.

\bibitem[Mur85]{Murre:Jacobian}
\bysame, \emph{Applications of algebraic {$K$}-theory to the theory of
  algebraic cycles}, Algebraic geometry, {S}itges ({B}arcelona), 1983, Lecture
  Notes in Math., vol. 1124, Springer, Berlin, 1985, pp.~216--261.

\bibitem[Nom95]{Noma}
Atsushi Noma, \emph{Hypersurfaces with smooth dual varieties}, Amer. J. Math.
  \textbf{117} (1995), no.~6, 1507--1515.

\bibitem[Sam60]{Samuel}
Pierre Samuel, \emph{Relations d'\'{e}quivalence en g\'{e}om\'{e}trie
  alg\'{e}brique.}, Proc. {I}nternat. {C}ongress {M}ath. 1958., ,, 1960,
  pp.~470--487.

\bibitem[Seg65]{Segre:Hermitian}
Beniamino Segre, \emph{Forme e geometrie hermitiane, con particolare riguardo
  al caso finito}, Ann. Mat. Pura Appl. (4) \textbf{70} (1965), 1--201.

\bibitem[She12]{Shen:Fermat}
Mingmin Shen, \emph{Rational curves on {F}ermat hypersurfaces}, C. R. Math.
  Acad. Sci. Paris \textbf{350} (2012), no.~15-16, 781--784.

\bibitem[Shi74]{Shioda:Fermat}
Tetsuji Shioda, \emph{An example of unirational surfaces in characteristic
  {$p$}}, Math. Ann. \textbf{211} (1974), 233--236.

\bibitem[Shi77]{Shioda:Unirationality}
\bysame, \emph{Some results on unirationality of algebraic surfaces}, Math.
  Ann. \textbf{230} (1977), no.~2, 153--168.

\bibitem[Shi01]{Shimada:Lattices}
Ichiro Shimada, \emph{Lattices of algebraic cycles on {F}ermat varieties in
  positive characteristics}, Proc. London Math. Soc. (3) \textbf{82} (2001),
  no.~1, 131--172.

\bibitem[SK79]{SK:Fermat}
Tetsuji Shioda and Toshiyuki Katsura, \emph{On {F}ermat varieties}, Tohoku
  Math. J. (2) \textbf{31} (1979), no.~1, 97--115.

\bibitem[{Stacks}]{stacks-project}
The {Stacks Project Authors}, \emph{\textit{Stacks Project}},
  {\scriptsize\url{https://stacks.math.columbia.edu}}.

\bibitem[Sta12]{Stanley:ECI}
Richard~P. Stanley, \emph{Enumerative combinatorics. {V}olume 1}, second ed.,
  Cambridge Studies in Advanced Mathematics, vol.~49, Cambridge University
  Press, Cambridge, 2012.

\bibitem[Tat65]{Tate:Conjecture}
John~T. Tate, \emph{Algebraic cycles and poles of zeta functions}, Arithmetical
  {A}lgebraic {G}eometry ({P}roc. {C}onf. {P}urdue {U}niv., 1963), Harper \&
  Row, New York, 1965, pp.~93--110.

\bibitem[Vol10]{Vollaard}
Inken Vollaard, \emph{The supersingular locus of the {S}himura variety for
  {${\rm GU}(1,s)$}}, Canad. J. Math. \textbf{62} (2010), no.~3, 668--720.

\bibitem[Wal56]{Wallace:Duality}
Andrew~H. Wallace, \emph{Tangency and duality over arbitrary fields}, Proc.
  London Math. Soc. (3) \textbf{6} (1956), 321--342.

\bibitem[Wei49]{Weil:Fermat}
Andr\'{e} Weil, \emph{Numbers of solutions of equations in finite fields},
  Bull. Amer. Math. Soc. \textbf{55} (1949), 497--508.

\bibitem[Wei52]{Weil:Grossenchar}
\bysame, \emph{Jacobi sums as ``{G}r\"{o}ssencharaktere''}, Trans. Amer. Math.
  Soc. \textbf{73} (1952), 487--495.

\end{thebibliography}
\end{document}